\documentclass[12pt,reqno,letterpaper]{amsart}
\usepackage{mathrsfs}

\usepackage[T1]{fontenc}
\usepackage{enumitem}
\usepackage{graphicx}
\graphicspath{ {./images/} }
\usepackage{esint}
\usepackage{hyperref}
\hypersetup{
  colorlinks   = true,
  citecolor    = blue,
  linkcolor    = blue
}

\usepackage{amsmath,amsfonts,amsthm,color,tikz, comment, xcolor, xfrac,amssymb}
\usepackage{mathtools}

\usepackage[normalem]{ulem}
\usepackage{mathrsfs}

\usepackage{pdfsync}
\usepackage[font={scriptsize}]{caption}
\usepackage{environ}

\usepackage[left=1in, right=1in, top=1in,bottom=1in]{geometry}
\numberwithin{equation}{section}

\def\XXint#1#2#3{{\setbox0=\hbox{$#1{#2#3}{\int}$ }
\vcenter{\hbox{$#2#3$ }}\kern-.6\wd0}}

\usepackage{fix-cm}

\newcommand{\A}{\mathbb A}
\newcommand{\C}{\mathbb C}

\newcommand{\R}{\mathbb R}

\newcommand{\T}{\mathbb T}
\newcommand{\Z}{\mathbb Z}
\renewcommand{\P}{\mathbb{P}}
\newcommand{\ep}{\varepsilon}

\newcommand{\norm}[1]{ \Vert #1 \Vert}

\newtheorem{theorem}{Theorem}[section]

\newtheorem{conjecture}[theorem]{Conjecture}

\newtheorem{definition}[theorem]{Definition}

\newtheorem{lemma}[theorem]{Lemma}

\newtheorem{proposition}[theorem]{Proposition}

\theoremstyle{remark}
\newtheorem{remark}[theorem]{Remark}

\theoremstyle{remark}

\newcommand{\bean}{\begin{eqnarray*}}
\newcommand{\eean}{\end{eqnarray*}}
\newcommand{\ben}{\begin{enumerate}}
\newcommand{\een}{\end{enumerate}}
\newcommand{\beq}{\begin{equation}}
\newcommand{\eeq}{\end{equation}}

\allowdisplaybreaks

\begin{document}

\author{Nicholas Gismondi, Kunyi(Mark) Ma, Mandon Pathak, Alexandru F. Radu}

\title{Non-unique solutions to the periodic \textsf{gKdV} equation}

\begin{abstract}
    In this paper we utilize a convex integration scheme to construct non-trivial weak solutions to the $k$-generalized \textsf{KdV} equation which lie in
    $$
    \bigcap_{\epsilon > 0} C_t^0 L_x^{k-\epsilon}([0,1] \times \T)
    $$
    and, when $k \ge 3$, it may also be chosen in
    \[
    \bigcap_{\epsilon >0} C_t^0 H_x^{\frac{1}{2} - \frac{1}{k} - \epsilon}([0,1] \times \T)
    \]
    attaining identically $0$ initial data. Since our solutions do not lie in $C_t^0 L_x^k$, this requires introducing a new notion of weak solution, which is in fact stronger than the classical notion of a weak solution when the nonlinearity is integrable. This result shows that a necessary condition for unconditional uniqueness for $k$-\textsf{gKdV} is that the nonlinearity lies in $C_t^0L^1_x$. In the case of \textsf{KdV} this is in fact also sufficient.
\end{abstract}
\maketitle

\section{Introduction}
\subsection{Motivation and Background}
The $k$-generalized Korteweg-de Vries ($k$-\textsf{gKdV}) equation with initial data $u_0$ and $k \geq 2$ is given by
\begin{equation}\label{eq:gkdv}
\begin{cases}
    \partial_t u + \partial_{xxx} u + \sigma\partial_x\left(u^k\right)= 0\\
    u(0,x) = u_0(x)
\end{cases}
\end{equation}
for $\sigma \in \{\pm 1\}$ and $u:[0,1] \times \T \to \R$. When $k$ is understood, we simply refer to $k$-\textsf{gKdV} as \textsf{gKdV}.

This formulation of the $\textsf{gKdV}$ equation includes several important special cases:
\begin{itemize}
    \item The case $k = 2$ with $\sigma = -1$ is known as the Korteweg-de Vries ($\textsf{KdV}$) equation, which is the most well-known member of this family. It was originally derived as a model for the unidirectional propagation of nonlinear dispersive long waves~\cite{KDV}. It is also discussed in context of inverse scattering.
    
    \item For $k = 3$, one obtains the modified \textsf{KdV} ($\textsf{mKdV}$) equation: the \textit{focusing} $\textsf{mKdV}$ equation when $\sigma = 1$, and the \textit{defocusing} $\textsf{mKdV}$ equation when $\sigma = -1$. \textsf{KdV} and \textsf{mKdV} are special in that they are completely integrable systems which possess infinitely many conserved quantities.
    
    \item Another important variant is the \textit{stationary} $\textsf{KdV}$ equation, which arises when time-independent solutions are considered. This reduction leads to an ordinary differential equation that plays a significant role in the study of traveling waves and integrable systems. This is discussed more in \cite{DN} and \cite{FN2016}. The third author recently applied the convex integration methodology to the stationary \textsf{KdV} in \cite{Pathak}; we will expand more on this in Section \ref{sec:6}. 
    
\end{itemize}

The Cauchy problem for \textsf{KdV} and \textsf{gKdV} equations has a large well-posedness theory. The early high-regularity theory goes back to Kato~\cite{Kato1983}. On the real line, Kenig-Ponce-Vega developed a systematic contraction-principle theory for \textsf{KdV} and generalized \textsf{KdV} equations, including local well-posedness and scattering results in Sobolev spaces~\cite{KPV1991,KPV1993}. Bourgain introduced the Fourier restriction norm method for periodic dispersive equations, including periodic \textsf{KdV}~\cite{BourgainKdV}; this method and its refinements became central to the low-regularity theory.

On the torus, the special integrable structure of \textsf{KdV} leads to results far below the energy space. Colliander-Keel-Staffilani-Takaoka-Tao proved sharp global well-posedness results for \textsf{KdV} and \textsf{mKdV} on both $\R$ and $\T$ in the Sobolev ranges where the corresponding local theories were known \cite{CKSTTSharp}. Babin-Ilyin-Titi later gave a direct unconditional uniqueness and Lipschitz-dependence result for periodic \textsf{KdV} in $\dot H^s(\T)$, $s\ge0$, in particular at the $L^2$ level modulo the spatial mean~\cite{BIT}. Related low regularity information for \textsf{mKdV} posed on $\R$ was obtained by Christ-Holmer-Tataru, who proved global-in-time a priori $H^s(\R)$ bounds in the range $-1/8<s<1/4$ \cite{CHT}. More recently, Molinet-Zhu proved that the $L^2(\R)$-unconditional global well-posedness mechanism persists for low-dispersion fractional \textsf{KdV} equations
\[
    \partial_t u-D_x^\alpha\partial_xu+\partial_x(u^2)=0,
\]
with $\alpha\in\left(\frac{55}{38},2\right]$ on the real line~\cite{MZ}.

For \textsf{KdV}, the endpoint picture is now particularly sharp. On the torus, Kappeler-Topalov proved global well-posedness in $H^{-1}(\T)$ by inverse spectral methods~\cite{KT}.  Killip-Vi\textcommabelow{s}an subsequently proved the sharp real-line result in $H^{-1}(\R)$ and their method also recovers the periodic $H^{-1}$ theorem~\cite{KV}. Molinet proved sharp ill-posedness results below this threshold for \textsf{KdV} and \textsf{mKdV} on the torus ~\cite{Molinet2011,Molinet2012}.

For the periodic monomial \textsf{gKdV} equation with $k \ge 2$, the low-regularity theory is more rigid. Colliander-Keel-Staffilani-Takaoka-Tao proved local well-posedness for periodic $k$-\textsf{gKdV} in $H^s(\T)$, $s \ge 1/2$, by working in Bourgain-type spaces ~\cite{CKSTT2004}. Bao-Wu obtained global well-posedness for periodic defocusing generalized \textsf{KdV} in $H^s(\T)$, with $s \ge 1/2$ in the quartic case and $s > 5/9$ in the quintic case~\cite{BaoWu}.

For equations with general analytic nonlinearity
\[
\partial_t u + \partial_{xxx} u +\partial_x f(u) = 0,
\]
Molinet-Tanaka recently proved unconditional local well-posedness in $H^s(\T)$ for every $s> 1/2$, and also derived global consequences at the energy level under their conservation-law hypotheses~\cite{MT}. In the mKdV case, stronger unconditional uniqueness results are known; in particular, Molinet-Pilod-Vento proved unconditional well-posedness for periodic \textsf{mKdV} in $H^s(\T)$, $s \ge 1/3$~\cite{MolinetPilodVento}.

The distinction between conditional and unconditional uniqueness is important for the present work. Many low-regularity well-posedness arguments construct solutions in auxiliary Bourgain-type resolution spaces. Such a theory gives solutions in $C_t^0H_x^s$, but uniqueness may initially be known only inside the smaller resolution class. By contrast, an unconditional theorem asserts uniqueness directly in the natural class $C_t^0H_x^s$, without imposing membership in an auxiliary space. This notion for dispersive PDEs was first introduced by Kato in the context of the nonlinear Schr\"{o}dinger equation \cite{Kato1995}. See also~\cite{BIT,MT,MolinetPilodVento}. 

Our result goes in the opposite direction. We prove an unconditional nonuniqueness statement in a weak singular class: the zero solution and the solution constructed in Theorem~\ref{thm:main} have the same initial data, and no auxiliary resolution space condition is imposed to distinguish them. Nonuniqueness occurs below the Lebesgue threshold at which the nonlinearity is classically meaningful. Indeed, if $u\in C_t^0L_x^k$, then $u^k \in C_t^0L_x^1$ and the term $\partial_x(u^k)$ is an ordinary distribution. Below $C_t^0L_x^k$, however, the product $u^k$ need not be integrable and one must specify what it means to solve the equation. The weak singular formulation below does this through an absolute Fourier summability condition on the nonzero modes of $u^k$.

In the \textsf{KdV} case, $k=2$, the result is sharp in the Lebesgue scale: Theorem ~\ref{thm:main} gives nonuniqueness in every $C_t^0L_x^{2-\epsilon}$, while at the $L^2$ endpoint ~\cite{BIT} shows unconditional uniqueness. For general $k\ge3$, the construction reaches the Sobolev endpoint-minus scale
\[
\bigcap_{\epsilon>0} C_t^0H_x^{\frac12-\frac1k-\epsilon}
\]
which is the Sobolev counterpart of the Lebesgue threshold $L^k_x$.

\subsection{Overview of the Scheme}\label{subsec_scheme}
To achieve this nearly optimal Sobolev regularity (at least in the case of \textsf{KdV}) for the unconditional nonuniqueness, we employ a convex integration scheme with intermittent building blocks which are nearly fully intermittent. Intermittency has played a key role in many convex integration constructions; see for instance  \cite{ABGN}, \cite{BBV}, \cite{BCV}, \cite{BMNV}, \cite{BV19}, \cite{CH}, \cite{CL2}, \cite{CheskidovLuo}, \cite{DS17}, \cite{GKN23}, \cite{Gismondi}, \cite{GR}, \cite{Luo}, \cite{MS}, \cite{NV}, \cite{Pathak}, \cite{Peng} and references therein. The point of this choice is so that the perturbations at level $q+1$, which we denote $w_{q+1}$, have order-one endpoint size, so that $w_{q+1}^k$ can cancel the old error, while their subendpoint Lebesgue and Sobolev norms are summable. More precisely, the intermittency parameters $\epsilon_{q+1}=2^{-(q+1)}$ tend to zero, so the perturbations approach the fully intermittent scaling
\[
\|w_{q+1}\|_{C_t^0L^p_x} \sim \lambda_{q+1}^{\frac1k-\frac1p}, \qquad p<k,
\]
up to a vanishing loss. At the same time, the frequency support of $w_{q+1}$ is constructed to be nearly the natural frequency scale $\lambda_{q+1}^{1+\epsilon_{q+1}}$, which gives
\[
\|w_{q+1}\|_{C_t^0H^\alpha_x} \lesssim \lambda_{q+1}^{(1+\epsilon_{q+1})\alpha}\lambda_{q+1}^{(1-\epsilon_{q+1})(\frac1k-\frac12)}.
\]
Hence the Sobolev increments are summable for every
\[
\alpha <\frac12-\frac1k
\]
and so the construction reaches the endpoint-minus Sobolev scale corresponding to the Lebesgue threshold $L^k_x$, while still keeping enough endpoint strength to cancel the error. In the \textsf{KdV} case $k=2$, this means that the nonuniqueness occurs below $L^2_x$, the endpoint where unconditional uniqueness is known.

The main new contribution introduced in this paper is to propose a very natural interpretation of the nonlinearity as an element of $C_t^0 \dot{H}_x^{-s}$ for some $s > 0$. To briefly summarize, we define $\mathbb{P}_{\not=0}(u^k)$ as an element of $C_t^0 \dot{H}_x^{-s}$ if the formal Fourier expansion is absolutely summable in this norm. Examining the expression this leads to, it is natural then to measure the errors generated by the scheme in what we refer to as \textit{homogeneous weighted Wiener spaces}, which we introduce in subsection \ref{sec:wws}. Morally speaking, we may consider these spaces to be the $\ell^1$ norm of the Fourier coefficients with the weight $|\xi|^{-s}$, which can be regarded as the $\ell^1$ counterpart of Sobolev spaces. Sobolev norms had previously been used as the norm to drive the errors to $0$ in the qualitative convex integration schemes used in \cite{ABGN}, \cite{CH}, \cite{Gismondi}, \cite{GR}, and \cite{Pathak}.

We mention here the result of Christ \cite{Christ1} who was able to show in an extended weak sense that solutions to the nonlinear Schr\"{o}dinger equation and \textsf{KdV} lying in $C_t^0 H_x^s([0,1] \times \T)$, $s < 0$, are not unique. To treat the issue of the nonlinearity not being integrable, Christ adopts a definition based on mollifying the solution $u$ and then requiring that when the mollification length scale is sent to $0$ the nonlinearity converges to a distribution. Our result is an improvement in two regards. First, our method will be able to produce integrable solutions and second, the notion of solution we introduce in Definition \ref{def:paras} and Definition \ref{def:weak_para_soln} is stronger than the Christ notion of solution. That is, the solutions we construct are also solutions in the Christ framework. We prove this in Lemma \ref{lem:absolute-implies-christ-cutoff} in the \hyperref[sec:appendix]{Appendix}. In addition, in \cite{ABGN}, \cite{CH}, \cite{Gismondi}, \cite{GR}, and \cite{Pathak} an alternative notion for defining the nonlinearity involving paraproducts is employed. We also show in Lemma \ref{lem:absolute-implies-paraproduct} in the \hyperref[sec:appendix]{Appendix} that the notion of solution we consider here is stronger than this paraproduct notion.

\subsection{Related open problems}

\;

The present work, together with the analysis of the stationary \textsf{KdV} equation carried out in \cite{Pathak}, may be viewed as part of a broader and \textsf{unified and methodological program} devoted to the \textit{development of convex integration techniques} in the context of \textit{general nonlinear dispersive PDEs}. Traditionally, convex integration has been employed with remarkable success in fluid dynamics, namely \cite{DLS09}, \cite{DLS13} and more recent papers in the context of Onsager program and intermittent convex integration, namely \cite{GKN23}, \cite{isett} and \cite{NV}. See also the papers referenced above on intermittency in the use of convex integration schemes. Its adaptation to dispersive equations, however, remains comparatively unexplored due to the presence of what we refer to as the dispersive term, $\partial_{xx}u$. Following \cite{Pathak} and \cite{Peng}, we may see that "standard intermittent convex integration schemes," like those encountered in \cite{BBV}, \cite{BCV}, \cite{BMNV}, \cite{BV19}, \cite{CL2}, \cite{CheskidovLuo}, \cite{DS17}, \cite{Luo}, \cite{MS}, \cite{NV}, and \cite{Peng} cannot succeed in this dispersive setting. For the purposes of this heuristic discussion, let us assume that $w_{q+1}$ is frequency localized to scale $\lambda_{q+1} \gg 1$. In these standard intermittent convex integration schemes one is required to estimate this term, $\partial_{xx}w_{q+1}$ either in an $L^p$ norm or a $C^\alpha$ norm \cite{BV2020}. 
Let us assume we work in the $L^1$ setting, with parameters given by
\[
\lambda_q = a^{b^{q}}, \qquad a,b>1, \qquad \delta_q = \lambda_q^{-k\beta}, \quad \beta>0.
\]
Suppose moreover that the inductive hypotheses furnish the bounds
\[
\|E_q\|_{C_t^0L^1_x} \lesssim \delta_{q+1}, \qquad \|w_q\|_{C^0_tL^1_x} \lesssim \delta_{q}^{\frac{1}{k}}\lambda_q^{\frac{1}{k}-1}
\]
Where $E_q$ is the "error" that one wishes to drive to zero in the convex integration scheme. From the intermittent scaling and the frequency localization of $w_{q+1}$ one has
$$
\Vert w_{q+1} \Vert_{C_t^0 H_x^\alpha} \lesssim \lambda_{q+1}^\alpha \Vert w_{q+1} \Vert_{C_t^0 L^2_x} \lesssim \lambda_{q+1}^{\alpha - \beta + \frac{1}{k} - \frac{1}{2}}
$$
and thus we would have the solution $u$ lies in $C_t^0 H^\alpha_x$ for $\alpha < \beta + \frac{1}{2} - \frac{1}{k}$. This surpasses the regularity achieved in our scheme (see Theorem \ref{thm:main}) so long as $\beta > 0$, but as we will demonstrate, this is impossible. To see this, we estimate the "dissipative error term" as
\begin{equation*}
\|\partial_{xx} w_{q+1}\|_{C_t^0L^1_x}
\lesssim \lambda_{q+1}^2 \|w_{q+1}\|_{C_t^0L^1_x} \lesssim \lambda_{q+1}^2 \, \delta_{q+1}^{\frac{1}{k}} \, \lambda_{q+1}^{\frac{1}{k}-1}.
\end{equation*}

For the inductive scheme to close, it is necessary that
\[
\lambda_{q+1}^2 \, \delta_{q+1}^{\frac{1}{k}} \, \lambda_{q+1}^{\frac{1}{k}-1}
\lesssim \delta_{q+2}.
\]
However, a comparison of the exponents shows that this requirement imposes the condition
\[
\beta \leq \frac{k+1}{k(1-bk)},
\]
which is strictly negative since $bk > k \geq 2$. This contradicts the assumption $\beta>0$, and hence the convex integration scheme cannot be closed under the given parameter regime.

Our scheme is successful so long as the nonlinearity $u^k \not \in C_t^0L^1_x$. If one wishes to push this scheme beyond this threshold, this would require one to use a "standard scheme" which as we've illustrated above cannot succeed. This leads one to conjecture that this natural threshold identified by the failure of the convex integration scheme is in fact the threshold for unconditional nonuniqueness vs. unconditional uniqueness. We formulate this conjecture below. More precisely, our considerations and the fact the result is known to be sharp for \textsf{KdV} indicate the possibility that the correct unconditional local well-posedness threshold for $k$-\textsf{gKdV} may be $s \geq \frac{1}{2} - \frac{1}{k}$.

\begin{conjecture}
    The $k$-generalized Korteweg-de Vries equation \eqref{eq:gkdv}, $k \geq 3$, is unconditionally locally well-posed in $H^s(\T)$ for all $s \geq \frac{1}{2} - \frac{1}{k}$.
\end{conjecture}

\subsection{Main Result}

As mentioned in the introduction, our notion of a weak solution must be able to handle the case when the nonlinearity $u^k \not \in C_t^0 L^1_x$. In Definition \ref{def:paras} we introduce the notion of how to interpret the nonlinearity as an element of $C_t^0 H_x^s$, and then in Definition \ref{def:weak_para_soln} we formulate what we refer to as a singular solution.

\begin{definition}[\textbf{Absolute Summability in $C_t^0 \dot H^s([0,1] \times \T)$}]\label{def:paras}
    Let $u_1,u_2,\ldots,u_k \in C_t^0 \mathcal{D}'([0,1] \times \T)$.  We say that $\mathbb{P}_{\not = 0}(u_1u_2 \cdots u_k)$ is well-defined as a distribution in $C_t^0 \dot{H}^s([0,1] \times \T)$ for some $s \in \R$ if
    $$
    \sum_{\xi_1 + \xi_2 + \ldots + \xi_k \not =0 } \Vert \hat{u}_1(t,\xi_1) \hat{u}_2(t,\xi_2) \cdots \hat{u}_k(t,\xi_k)\Vert_{C_t^0} |\xi_1 + \xi_2 + \ldots + \xi_k|^{s} < \infty
    $$
    Then we define
    $$
        \mathbb{P}_{\not = 0}(u_1u_2\cdots u_k) := \sum_{\xi_1 + \xi_2 + \ldots + \xi_k \not=0} \hat{u}_1(t,\xi_1) \hat{u}_2(t,\xi_2) \cdots \hat{u}_k(t,\xi_k) e^{2\pi i (\xi_1 + \xi_2 + \ldots + \xi_k)x}
    $$
    since the right-hand side is an absolutely summable series in $C_t^0 \dot H^s([0,1] \times \T)$\footnote{Definition \ref{def:paras} is based on \cite[Definition 1]{LR} and discussion with Matthew Novack.}.
\end{definition}

\begin{definition}[\textbf{Weak singular solutions to ~\eqref{eq:gkdv}}] \label{def:weak_para_soln}
    We say $u \in C_t^0 H^s_x([0,1] \times \T)$ is a weak singular solution \footnote{The name \textit{singular solution} is based on the name introduced in \cite{CH}.} to the \textsf{gKdV} equation if there is $s' \in \R$ such that $\mathbb{P}_{\not = 0}(u^k)$ is well defined as an element of $C_t^0\dot{H}_x^{s'}$ in the sense of Definition \ref{def:paras} and
    $$
    \int_0^1 \left(\langle u,\partial_t \phi \rangle_{H^s,H^{-s}} + \langle u, \partial_{xxx}\phi \rangle_{\dot{H}^s,\dot{H}^{-s}} + \sigma \left\langle \mathbb{P}_{\not=0}\left(u^k\right), \partial_x \phi \right\rangle_{\dot{H}^{s'},\dot{H}^{-s'}}\right)\, dt =  -\langle u_0, \phi(0) \rangle_{H^s,H^{-s}}.
    $$
    for all $\phi \in C^\infty_{c,t} H^{s}_x([0,1) \times \T)$, where $C_c^\infty$ denotes the space of smooth compactly supported functions.

\end{definition}

With this, we may state our main result.

\begin{theorem}[\textbf{Non-Uniqueness}]\label{thm:main}
     For $k=2$ and $s > 3$, there is
     $$
     u \in \bigcap_{\epsilon>0} C_t^0L^{k-\epsilon}_x([0,1] \times \T) \setminus \{0\}
     $$
     such that $\mathbb{P}_{\not=0}(u^2)$ exists as an element of $C_t^0\dot{H}_x^{-s}([0,1] \times \T)$ in the sense of Definition \ref{def:paras}, and which solves ~\eqref{eq:gkdv} in the sense of Definition \ref{def:weak_para_soln} and satisfies $u(0,x) = 0$.
     
     Similarly, for $k \geq 3$ and $s > 3$, there is
     $$
     u \in \bigcap_{\epsilon>0} C_t^0H^{\frac{1}{2}- \frac{1}{k}-\epsilon}_x([0,1] \times \T) \setminus \{0\}
     $$
     such that $\mathbb{P}_{\not=0}(u^k)$ exists as an element of $C_t^0\dot{H}_x^{-s}([0,1] \times \T)$ in the sense of Definition \ref{def:paras}, and which solves ~\eqref{eq:gkdv} in the sense of Definition \ref{def:weak_para_soln} and satisfies $u(0,x) = 0$.
\end{theorem}

\begin{remark}
    In Theorem \ref{thm:main} in the case when $k=2$, one is able to in addition arrange that
    $$
    u \in \bigcap_{\epsilon > 0} C_t^0 B^{-s}_{\infty,\infty,x}([0,1] \times \T)
    $$
    so that the solution is both "nearly square integrable" and "nearly continuous." In section \ref{subsec:intermittency} we will discuss how this refinement can be achieved.
\end{remark}

\begin{remark}
    We remark here that $k$-\textsf{gKdV} always has the following conserved quantities for sufficiently regular solutions:
    \begin{align*}
        \text{Mass:}& \qquad \int_{\T} u(t,x)\, dx\\
        \text{Momentum:}& \qquad \int_{\T} |u(t,x)|^2\, dx\\
        \text{Energy:}& \qquad \int_{\T} \left(|\partial_x u(t,x)|^2 + \frac{1}{k+1} \left(u(t,x)\right)^{k+1}\right)\, dx.
    \end{align*}
    \textsf{KdV} and \textsf{gKdV}, being completely integrable, come with infinitely many more conserved quantities, but we do not pursue this further. Notice in order for the energy to be a well-defined quantity, this requires that $u \in C_t^0H^1_x$, but the solutions constructed in Theorem \ref{thm:main} do not achieve this regularity and thus the energy of the solutions we construct is undefined. When $k \geq 3$, the solutions we construct will lie in $C_t^0 L^2_x$, and thus the momentum is well-defined. However, this will not be a conserved quantity, as otherwise it would force all solutions which have identically zero initial data to remain identically zero for all time. Finally, for all $k \geq 2$, the solutions we construct will lie in $C_t^0 L^1_x$, and hence the mass is a well-defined quantity, and indeed it is also a conserved quantity. We prove this in Lemma \ref{lem:mean-conservation} in the \hyperref[sec:appendix]{appendix}. This explains why we insist on the solutions we construct having zero spatial mean for all time, it is in fact forced by the way we define a weak singular solution with identically zero intial data.
\end{remark}

The rest of this paper is concerned with the proof of Theorem \ref{thm:main}.

\subsection{Outline}

The paper is organized as follows.  In Section \ref{sec:2} we set up the analytic
framework used in the iteration.  This includes the Littlewood-Paley notation,
the weighted Wiener spaces in which the error is measured, and the intermittent
profiles which are used to generate large $k$-th powers while remaining small
in subcritical norms. We also record the elementary multiplier and convolution
estimates needed later in the proof.

Section \ref{sec:3} introduces the relaxed equation and describes one step of the convex integration scheme. Starting from a relaxed solution $(u_q,E_q)$, we add an intermittent perturbation $w_{q+1}$ and decompose the new error into its oscillation, Nash, dispersion, and temporal parts.

Section \ref{sec:4} states the inductive proposition. This is the quantitative core of
the paper: it records the support, convergence, error, and absolute-summability
bounds propagated by the iteration.

Section \ref{sec:5} proves the inductive proposition. We construct the increment $w_{q+1}$ and estimate the four errors produced by the iteration. We then verify the absolute Fourier summability of the nonlinear products and pass to the limit to prove Theorem~\ref{thm:main}.

Section \ref{sec:6} contains further consequences and variants.  We first discuss the stationary equation, where the endpoint $L^k$ class is rigid but the
subendpoint classes remain flexible. We then explain why the same construction
also applies to the generalized Benjamin-Ono and to other translation-invariant
dispersive variants for which the linear term contributes only a controllable
dispersion error.

The \hyperref[sec:appendix]{appendix} collects compatibility properties of the weak singular formulation.
First, we prove that weak singular solutions conserve their spatial mean. We
then compare the absolute Fourier summability condition with two more standard
ways of interpreting rough nonlinearities: Littlewood-Paley paraproduct
summation and Christ's cutoff formulation. Finally, we show that whenever the
regularity is high enough to make the ordinary product $u^k$ belong to $L^1$, the weak singular formulation implies the classical distributional weak formulation.

\subsection{Acknowledgments}
 M.P. was partially supported by NSF 2400238. A.R. was partially supported by a grant of the Ministry of Research, Innovation and Digitization, CCCDI - UEFISCDI, project number ROSUA-2024-0001, within PNCDI IV.

\section{Background and Tools}\label{sec:2}

The following on Littlewood-Paley theory, Fourier multiplier operators, and Sobolev spaces can be found in \cite{BCD} and \cite{Grafakos}.

\begin{definition}[\textbf{Littlewood-Paley projectors}]\label{def:projs}
    There exists $\varphi \colon \R\to[0,1]$, smooth, radially symmetric, and compactly supported in $\{\xi : 6/7 \leq |\xi|\leq 2\}$ such that $\varphi(\xi) = 1$ on $\{\xi : 1 \leq |\xi| \leq 12/7\}$,
    \begin{equation}
        \sum_{j\geq 0}\varphi(2^{-j}\xi)=1 \hspace{0.25cm} \text{ for all } \hspace{0.25cm} |\xi|\geq 1, \notag
    \end{equation}
    and $\operatorname{supp}\varphi_{j}\cap \operatorname{supp} \varphi_{j'}=\emptyset$ for all $|j-j'|\geq 2$, where $\varphi_j(\cdot) = \varphi(2^{-j}\cdot)$. We define the projection of a function $f:[0,1] \times \T \to \C$ on its $0$-mode by
    \begin{equation}
        \mathbb{P}_{=0,x}(f)(t)=\int_{\T} f(t,x)\, dx, \notag
    \end{equation}
    and the projection on the $j^{\rm th}$ shell by
    \begin{equation}
        \mathbb{P}_{2^j,x}(f)(t,x)=\sum_{\xi\in \Z}\hat{f}(t,\xi)\varphi_{j}(\xi)e^{2\pi i\xi x} \, . \notag
    \end{equation}
    We also define the frequency cutoff by
    $$
    \mathbb{P}_{\leq 2^j,x}(f) = \mathbb{P}_{=0,x}(f) + \sum_{k=1}^j \mathbb{P}_{2^k,x}(f),
    $$
    the projection onto high frequencies by $\mathbb{P}_{>2^j,x}(f) = (\operatorname{Id} - \mathbb{P}_{\leq 2^j,x})(f)$ and the projection off the mean by $\mathbb{P}_{\neq 0,x}f:=(\operatorname{Id}-\mathbb{P}_{=0,x})f$. 
    Denote by $m_{\le 2^j}$ the Fourier multiplier of $\mathbb P_{\le 2^j,x}$. Since $\varphi\ge0$, the above definition gives 
    \[
    0\le m_{\le 2^j}(\xi)\le 1
    \]
    for every $\xi\in\mathbb Z$.
\end{definition}

\begin{remark}
    In Definition \ref{def:projs} we have included the variable $x$ in the subscripts of the various operators to emphasize that the operators are solely acting on the spatial variables. Going forward when it is clear, we will omit the $x$ in the subscript and merely write $\mathbb{P}_{=0}$ instead of $\mathbb{P}_{=0,x}$, $\mathbb{P}_{\not=0}$ instead of $\mathbb{P}_{\not=0,x}$, etc.
\end{remark}

\begin{lemma}[\textbf{$L^p$ boundedness of projection operators}]\label{lem:proj}
    $\mathbb{P}_{\leq \lambda}$ is a bounded operator from $L^p$ to $L^p$ for $1 \leq p \leq \infty$ with operator norm independent of $\lambda$.\footnote{This is a quick consequence of the Poisson summation formula.}
\end{lemma}

\begin{definition}[\textbf{$\dot H^s$ Sobolev spaces}] For $s \in \mathbb{R}$, we define
    $$\dot{H}^{s}(\mathbb{T})=\left\{ f: \sum_{\xi \in \Z \setminus \{0\}}|\xi|^{2s}|\hat{f}(\xi)|^2<\infty\right\}  $$
    with the norm induced by the sum above.
\end{definition}

\begin{definition}[\textbf{Fourier Coefficients}]
    For $f \in C_t^0 \mathcal{D}'_x$ we define the Fourier coefficients by
    $$
    \hat{f}(t,\xi) = \langle f(t,x), e^{-2\pi i \xi x}\rangle.
    $$
    For $f\in C_t^0L^1_x$ this clearly is equivalent to
    \[\hat{f}(t,\xi)=\int_{\T}f(t,x)e^{-2\pi i\xi x}\,dx, \hspace{0.25cm} \text{ where } \hspace{0.25cm} \T=[0,1].\]
\end{definition}

% \begin{definition}{\textbf{(Homogeneous Besov spaces)}}
%     For $\alpha \in \R$, $0<p,q\leq \infty$, and $f \in \mathcal{D}'(\T)$, define the homogeneous Besov space to be
%     \[
%         \dot{B}^{\alpha}_{p,q} = \left\{ f: \|f\|_{\dot{B}^{\alpha}_{p,q}} = \left\| 2^{j\alpha}\|\mathbb{P}_{2^j}f\|_{L^p(\T)}\right\|_{\ell^q_{j\geq 0}} <\infty \right\}.
%     \]
% \end{definition}
% % \vspace{-2.5em}
% \begin{remark}
%     Note for $\alpha > 0$ the space $\dot{B}^{\alpha}_{\infty,\infty}$ is equivalent to $\{f \in C^\alpha : \hat{f}(0) = 0\}$. In addition, the space $\dot{B}^{\alpha}_{2,2}$ is equivalent to the homogeneous Sobolev space $\dot{H}^\alpha$ for all $\alpha \in \R$.
% \end{remark}

\subsection{Weighted Wiener Spaces}\label{sec:wws}

We next introduce the weighted Wiener spaces, which are Fourier-side \(\ell^1\) spaces with polynomial
weights.  The main advantage of working with Wiener-type norms is that they
allow us to estimate oscillatory errors directly by summing their Fourier
coefficients, and in particular, negative weights are well suited for measuring
high-frequency errors.

\begin{definition}[\textbf{(Homogeneous Weighted Wiener Space)}] For $s\in \R$, $f \in \mathcal{D}'(\T)$ we define the \textit{homogeneous Wiener space} to be
$$
\dot{\mathbb{A}}^s(\T) = \left\{f : \sum_{\xi \not = 0} |\xi|^{s}|\hat{f}(\xi)| < \infty\right\}.
$$
Define 
\[\norm{f}_{\dot{\A}^{s}} = \sum_{\xi \neq 0} |\xi|^s |\hat{f}(\xi)|.\]
It is clear $\dot{\A}^s(\T) \cong
\ell^1\bigl(\{\xi\neq 0\}, |\xi|^s\bigr)$ modulo constants, is a Banach Space.
\end{definition}

To encode the time dependence, we propose the alternative space as follows.

\begin{definition}[\textbf{Spacetime Homogeneous Weighted Wiener Space}]
    For $s \in \R$, define the spacetime homogeneous weighted Wiener space as 
    $$
\dot{\mathbb{A}}_{x,t}^{s}([0,1] \times \T) = \left\{u \in \mathcal{D}^\prime([0, 1] \times \T) :  \sum_{\xi \not =0} |\xi|^s \Vert \hat{u}(t,\xi) \Vert_{C_t^0} < \infty\right\}
$$ and define 
\[\norm{u}_{\dot{\mathbb{A}}_{x,t}^{s}} = \sum_{\xi \neq 0} |\xi|^s \Vert \hat{u}(t,\xi) \Vert_{C_t^0}.\]
It is clear $\dot{\mathbb{A}}_{x,t}^{s}([0,1] \times \T)  \cong \ell^1(\{\xi\neq 0\}, |\xi|^s; C^0([0, 1]))$ modulo $C^0([0, 1])$-valued spatial constants is a Banach Space.
\end{definition}

\begin{remark}
    Note the spaces $\dot{\mathbb{A}}_{x,t}^{s}([0,1] \times \T)$ and $C_t^0\dot{\A}_x^s = C^0([0, 1] ; \dot{\A}^{s}(\T))$ equipped with norm 
    \[\norm{u}_{C_t^0 \dot{\A}_x^s} = \underset{t \in [0, 1]}{\sup} \sum_{\xi \neq 0} |\xi|^s |\hat{u}(t,\xi)|  \] are different in nature. It is obvious that one has the continuous embedding $\dot{\mathbb{A}}_{x,t}^{s} \hookrightarrow C_t^0\dot{\A}_x^s$. Note that the inclusion is strict.
\end{remark}

\begin{definition}[\textbf{Inhomogeneous Weighted Wiener Spaces}] For $s\in \R$, $f \in \mathcal{D}'(\T)$ we define the \textit{inhomogeneous Wiener space} to be
$$
\mathbb{A}^s(\T) = \left\{f \in \mathcal{D}'(\T) : \sum_{\xi \in \Z} \langle \xi\rangle^{s}|\hat{f}(\xi)| < \infty\right\}.
$$
where 
\[\langle \xi\rangle = (1 + |\xi|^2)^{\frac{1}{2}}\]
Define 
\[\norm{f}_{\A^{s}} =\sum_{\xi \in \Z} \langle \xi\rangle^{s}|\hat{f}(\xi)|\]
Similarly, define the spacetime version as 
   $$
\mathbb{A}_{x,t}^{s}([0,1] \times \T) = \left\{u \in C_t^0\mathcal{D}^\prime_x([0, 1] \times \T) :  \sum_{\xi \in \Z} \langle \xi\rangle^s \Vert \hat{u}(t,\xi) \Vert_{C_t^0} < \infty\right\}
$$ and define 
\[\norm{u}_{\mathbb{A}_{x,t}^{s}} = \sum_{\xi \in \Z} \langle\xi\rangle^s \Vert \hat{u}(t,\xi) \Vert_{C_t^0}\]
\end{definition}
As above, \(\mathbb A^s(\mathbb T)\) is isometrically isomorphic to the
weighted sequence space
\(\ell^1(\mathbb Z,\langle\xi\rangle^s)\), and
\(\mathbb A_{x,t}^s([0,1]\times\mathbb T)\) is isometrically isomorphic to
\(\ell^1(\mathbb Z,\langle\xi\rangle^s;C^0([0,1]))\). Hence both spaces are
Banach spaces.

\begin{lemma}[\textbf{Kato-Ponce-type product estimate for weighted Wiener Spaces}]\label{lem:pseudo_diff_cont_Wiener}
    Let $\alpha,\beta \in C^{\infty}(\T)$ with $\beta$ having zero mean. Then for $s \geq 0$ we have
    $$
        \Vert \alpha \beta \Vert_{\dot{\A}^{-s}} \lesssim \Vert \alpha \Vert_{\A^s} \Vert \beta \Vert_{\dot{\A}^{-s}}.
    $$
    Moreover, for 
\[
    \alpha\in \mathbb A_{x,t}^{s}([0,1]\times\mathbb T),
    \qquad
    \beta\in \dot{\mathbb A}_{x,t}^{-s}([0,1]\times\mathbb T),
\]
and assume that \(\beta(t,\cdot)\) has zero spatial mean for every \(t\in[0,1]\). Then
\[
    \|\alpha\beta\|_{\dot{\mathbb A}_{x,t}^{-s}}
    \lesssim_s
    \|\alpha\|_{\mathbb A_{x,t}^{s}}
    \|\beta\|_{\dot{\mathbb A}_{x,t}^{-s}}.
\]
\end{lemma}

\begin{proof}
    Without loss of generality we show for spacetime version.
Since \(\beta(t,\cdot)\) has zero spatial mean, we have
\[
    \widehat\beta(t,0)=0
    \qquad
    \text{for all }t\in[0,1].
\]
Therefore, for every \(\xi\neq0\),
\[
    \widehat{\alpha\beta}(t,\xi)
    =
    \sum_{\eta\in\mathbb Z}
    \widehat\alpha(t,\xi-\eta)\widehat\beta(t,\eta)
    =
    \sum_{\eta\neq0}
    \widehat\alpha(t,\xi-\eta)\widehat\beta(t,\eta).
\]
Hence
\[
\begin{aligned}
\|\alpha\beta\|_{\dot{\mathbb A}_{x,t}^{-s}}
&=
\sum_{\xi\neq0}
|\xi|^{-s}
\left\|
    \widehat{\alpha\beta}(\cdot,\xi)
\right\|_{C_t^0}
\\
&\leq
\sum_{\xi\neq0}
|\xi|^{-s}
\left\|
    \sum_{\eta\neq0}
    \widehat\alpha(\cdot,\xi-\eta)
    \widehat\beta(\cdot,\eta)
\right\|_{C_t^0}
\\
&\leq
\sum_{\xi\neq0}
\sum_{\eta\neq0}
|\xi|^{-s}
\|\widehat\alpha(\cdot,\xi-\eta)\|_{C_t^0}
\|\widehat\beta(\cdot,\eta)\|_{C_t^0}.
\end{aligned}
\]
For \(s\geq0\) and \(\xi,\eta\neq0\), we use
\[
    |\xi|^{-s}
    \lesssim_s
    |\eta|^{-s}
    \langle \xi-\eta\rangle^s.
\]
Indeed,
\[
    |\eta|
    \leq |\xi|+|\xi-\eta|
    \leq |\xi|(1+|\xi-\eta|),
\]
because \(|\xi|\geq1\). Therefore
\[
    |\xi|^{-s}
    \leq
    |\eta|^{-s}(1+|\xi-\eta|)^s
    \lesssim_s
    |\eta|^{-s}\langle \xi-\eta\rangle^s.
\]
Thus
\[
\begin{aligned}
\|\alpha\beta\|_{\dot{\mathbb A}_{x,t}^{-s}}
&\lesssim_s
\sum_{\xi\neq0}
\sum_{\eta\neq0}
|\eta|^{-s}
\langle \xi-\eta\rangle^s
\|\widehat\alpha(\cdot,\xi-\eta)\|_{C_t^0}
\|\widehat\beta(\cdot,\eta)\|_{C_t^0}
\\
&=
\sum_{\eta\neq0}
|\eta|^{-s}
\|\widehat\beta(\cdot,\eta)\|_{C_t^0}
\sum_{\xi\neq0}
\langle \xi-\eta\rangle^s
\|\widehat\alpha(\cdot,\xi-\eta)\|_{C_t^0}.
\end{aligned}
\]
For fixed \(\eta\), setting \(m=\xi-\eta\), we have
\[
    \sum_{\xi\neq0}
    \langle \xi-\eta\rangle^s
    \|\widehat\alpha(\cdot,\xi-\eta)\|_{C_t^0}
    \leq
    \sum_{m\in\mathbb Z}
    \langle m\rangle^s
    \|\widehat\alpha(\cdot,m)\|_{C_t^0}
    =
    \|\alpha\|_{\mathbb A_{x,t}^s}.
\]
Consequently,
\[
\begin{aligned}
\|\alpha\beta\|_{\dot{\mathbb A}_{x,t}^{-s}}
&\lesssim_s
\|\alpha\|_{\mathbb A_{x,t}^s}
\sum_{\eta\neq0}
|\eta|^{-s}
\|\widehat\beta(\cdot,\eta)\|_{C_t^0}
\\
&=
\|\alpha\|_{\mathbb A_{x,t}^s}
\|\beta\|_{\dot{\mathbb A}_{x,t}^{-s}}.
\end{aligned}
\]
\end{proof}

\subsection{Intermittent Slabs}\label{subsec:intermittency}

As mentioned previously, intermittency plays a crucial role in the construction. Roughly speaking, an intermittent function is one which is highly concentrated in space and whose $L^p$ norms obey radically different estimates. For the same reasons discussed in \cite{Gismondi} and \cite{Pathak}, we utilize a generic intermittent building block which we do not require to be a stationary solution of any PDE. However, unlike these schemes, one must utilize the maximum amount of intermittency in order to achieve the desired Sobolev estimates; this is more reminiscent of what is presented in \cite{GR}. The reason for this is clear: The intermittent slabs constructed in Lemma \ref{lem:boldW} will be $L^k$-normalized, and so for $k \geq 3$, to achieve the maximum $L^2$-based decay one must saturate Bernstein's inequality. In the language of Lemma \ref{lem:boldW}, this means sending $\epsilon \to 0$. This provides "room" for a positive amount of fractional derivatives to be taken and maintain decay. However, when $k = 2$, the intermittent slab is $L^2$ normalized, so therefore any amount of fractional anti-derivatives will give decay. For this reason, in the \textsf{KdV} case, one may instead choose to utilize a vanishingly small amount of intermittency in order to achieve in some sense the ideal amount of Besov regularity allowed by the scheme. We will not pursue this here, but the scheme utilized in \cite{Pathak} can easily be adapted to the time dependent \textsf{KdV} setting.

\begin{lemma}[\textbf{$L^k$ normalized intermittent slabs}]\label{lem:boldW}
    Fix $k \geq 2$. Let $\lambda$ be a large power of $2$ and take $0 < \epsilon < 1$ such that $\lambda^{\epsilon}$ is also an integer. Then there exist smooth $\rho_{k,\lambda,\epsilon}:\T \to \R$ such that
    \begin{enumerate}
        \item\label{w:1} $\int_{\T} \rho_{k,\lambda,\epsilon} = 0$ and $\int_\T \rho_{k, \lambda, \epsilon}^k = 1$;
        \item\label{w:2} $\Vert \rho_{k,\lambda,\epsilon} \Vert_{L^p(\T)} \lesssim \lambda^{(1-\epsilon)\left(\frac{1}{k}-\frac{1}{p}\right)}$;
        \item\label{w:3} $\rho_{k,\lambda,\epsilon}$ is $\frac{\T}{\lambda^{\epsilon}}$-periodic.
    \end{enumerate}
\end{lemma}
\begin{proof}
For fixed $k\ge 2$, choose a nonzero real-valued $\psi_k \in C_c^\infty(-1/2,1/2)$ with 
\[
\int_{\R} \psi_k = 0.
\]
Set
\[
\widetilde\psi_k(x):=\psi_k(-x), \qquad \phi:=\psi_k \ast \widetilde\psi_k.
\]
Then $\phi\in C_c^\infty(-1,1)$,
\[
\int_{\R}\phi=0, \qquad \widehat\phi(\xi)=|\widehat{\psi_k}(\xi)|^2\ge0.
\]
Since $\widehat\phi\ge0$ and $\phi\not\equiv0$, we have
\[
\int_{\R}\phi^k >0.
\]
We normalize
\[
\phi_k:=\frac{\phi}{\left(\int_{\R}\phi^k\right)^{1/k}}
\]
Thus
\[
\phi_k\in C_c^\infty(-1,1), \qquad \int_{\R}\phi_k=0, \qquad \int_{\R}\phi_k^k=1, \qquad \widehat{\phi_k}\ge0.
\]
Then define 
    \begin{equation}\label{eq:rho}
        \rho_{k, \lambda, \epsilon}(x) = \sum_{n \in \Z} \lambda^{(1 - \epsilon)\frac{1}{k}}\phi_k(\lambda x+ \lambda^{1  - \epsilon} n)
    \end{equation}
    Using the number of active Fourier coefficients is $\#\{n \mid |\lambda x + \lambda^{1 - \epsilon}n| \lesssim 1\} \sim \lambda^{\epsilon}$, and support size $\lambda^{-1}$, one has 
    \begin{align*}
        \norm{\rho_{k, \lambda, \epsilon}}_{L^p}^p &\lesssim \lambda^{\epsilon} \cdot  \lambda^{(1 - \epsilon)\frac{p}{k}} \cdot \lambda^{-1} \\
         \norm{\rho_{k, \lambda, \epsilon}}_{L^p}&\lesssim  \lambda^{(1 - \epsilon)(\frac{1}{k} - \frac{1}{p})}
    \end{align*}
\end{proof}

\begin{lemma}[\textbf{Fourier series representation of the intermittent slab}]\label{lem:fourier}
    The Fourier series representation of $\rho_{k,\lambda,\epsilon}$ is
    \begin{equation}\label{eq:fourier_series}
        \rho_{k,\lambda,\epsilon}(x) = \sum_{n \in \Z} \lambda^{(1-\epsilon)\left(\frac{1}{k} - 1\right)} \hat{\phi}_k\left(\lambda^{\epsilon-1}n\right) e^{2\pi i \lambda^{\epsilon}  n x}.
    \end{equation}
    In particular 
    \begin{equation}\label{eq:fourier_series rho k}
        \rho_{k, \lambda, \epsilon}^k(x) = \sum_{n \in \Z} \hat{g}_k(\lambda^{\epsilon - 1} n) e^{2\pi i \lambda^\epsilon n x}.
    \end{equation}
    for some $g_k \in C_c^\infty(-1, 1)$, and $\int_\T g_k = 1$.
\end{lemma}
\begin{proof}
    Applying the Poisson summation formula to ~\eqref{eq:rho} gives ~\eqref{eq:fourier_series}. Now define $g_k = \phi_k^k$ and using disjoint support for large $\lambda$ one obtain 
    \[\rho_{k, \lambda, \epsilon}^k(x) = \sum_{n \in \Z} \lambda^{1 - \epsilon} g_k(\lambda x + \lambda^{1 - \epsilon}n)\]
    Thus applying Poisson summation yields the result.
\end{proof}

\begin{lemma}[\textbf{Projections onto frequencies larger than maximum effective frequency)}]\label{lem:freq_proj_rho}
    Set $\nu = \lambda^{1 + \epsilon}$. Then
    $$
    \Vert \mathbb{P}_{\geq \nu}(\rho_{k,\lambda,\epsilon}) \Vert \lesssim \lambda^{\frac{1-\epsilon}{k}} \lambda^{-N\epsilon}
    $$
    for any $N > 0$. The implicit constant in the inequality above depends on $N$, $\epsilon$, and $\nu$, but not $\lambda$.
\end{lemma}
\begin{proof}
    The proof of this follows \cite[Lemma 2.10]{Pathak} with the necessary modifications for the frequency truncation, setting $a = 1$, and the new intermittent slabs deployed here, as well as carefully tracking the asymptotics in $\lambda$. We leave the details to the reader.
\end{proof}

\section{Convex Integration Scheme}\label{sec:3}
We now describe the convex integration mechanism. The starting point is the relaxed version of \eqref{eq:gkdv}, 
\begin{equation}\label{eq:relaxed}
    \partial_t u + \partial_{xxx} u + \sigma \partial_x\left(u^k\right) = \partial_x E,
\end{equation}
where $E:[0,1]\times\T\to\R$ is the error. The goal of the iteration is to construct smooth pairs $(u_q,E_q)$ solving \eqref{eq:relaxed}, while driving $E_q\to0$ in a weighted Wiener norm with a negative regularity index and keeping $u_q$ convergent in the desired subcritical Sobolev spaces.

Assume that $(u_q,E_q)$ is already a smooth solution of \eqref{eq:relaxed}. Define
\begin{equation}\label{eq:uq+1}
    u_{q+1}=u_q+w_{q+1},
\end{equation}
where $w_{q+1}$ is a highly oscillatory intermittent perturbation with zero mean. This is chosen so that the mean of its $k^{th}$ power cancels the old error, $E_q$, possibly with some remainder terms which are small in the weighted Wiener spaces.

% More precisely, we take
% \[
% w_{q+1}= \mathbb P_{\le \mu_{q+1}^{1/2}}a_q \mathbb P_{\le \nu_{q+1}} \rho_{q+1} h_{q+1},
% \]
% where
% \[
% a_q = (A_{q+1}-\sigma E_q)^{1/k}.
% \]
% The parameter $A_{q+1}$ is chosen large enough so that this amplitude is smooth and positive. The function $\rho_{q+1}$ is the itermittent slab from Lemma~\ref{lem:boldW}, normalized by
% \[
% \int_\T \rho_{q+1}=0, \qquad \int_\T \rho_{q+1}^k=1.
% \]
% The cutoff $h_{q+1}$ is identically $1$ on the temporal support of the previous stage and slightly enlarges the support of the solution. The frequency parameters are chosen lacunary, with
% \[
% \epsilon_{q+1}= 2^{-(q+1)}, \qquad \mu_{q+1}=\lambda_{q+1}^{\epsilon_{q+1}}, \qquad \nu_{q+1}= \lambda_{q+1}^{1+\epsilon_{q+1}}.
% \]

More precisely, expanding the relaxed equation for $u_{q+1}$, we obtain 
\begin{equation*}
    \begin{split}
        \partial_x E_{q+1} &= \partial_t u_{q+1} + \partial_{xxx} u_{q+1} + \sigma \partial_x \left(u_{q+1}^k\right)\\
        &= \partial_t w_{q+1} + \partial_x \left(E_q + \sigma w_{q+1}^k\right) + \partial_{xxx} w_{q+1} + \sigma \partial_x\left(\sum_{n=1}^{k-1} \binom{k}{n} w_{q+1}^n u_q^{k-n} \right)
    \end{split}
\end{equation*}
Since $w_{q+1}$ has zero spatial mean, let $\omega_{q+1}$ denote its mean-free antiderivative
\[
\partial_x\omega_{q+1}=w_{q+1}.
\]
We then define
\begin{equation}\label{eq:osc_error}
    E_O = E_q + \sigma w_{q+1}^k,
\end{equation}
\begin{equation}\label{eq:Nash_error}
    E_N = \sigma \sum_{n=1}^{k-1} \binom{k}{n} w_{q+1}^n u_q^{k-n}
\end{equation}
\begin{equation}\label{eq:dis_error}
    E_D = \partial_{xx} w_{q+1}
\end{equation}
and
\begin{equation}\label{eq:temporal_error}
    E_T = \partial_t \omega_{q+1}.
\end{equation}
Thus
\begin{equation}\label{eq:Rq+1}
    E_{q+1} = E_O + E_N + E_D + E_T.
\end{equation}
We refer to these four terms as the oscillation error, Nash error, dispersion error and temporal error, respectively. 

The oscillation error is the main cancellation. We will use that
$$
\sigma\mathbb{P}_{=0}(w_{q+1}^k) = -E_q + R_{q+1}
$$
where $R_{q+1}$ is some remainder term which can be made arbitrarily small in $\dot{\mathbb{A}}^{-s}_{x,t}$ for some $s > 0$ independent of $q$ by choosing the frequency of $w_{q+1}$ large enough. Thus
$$
E_O = R_{q+1} + \sigma\mathbb{P}_{\not=0}(w_{q+1}^k).
$$
and finally using that the minimum frequency of $\mathbb{P}_{\not=0}(w_{q+1}^k)$ is very large, we will be able to show this term can be made arbitrarily small in $\dot{\mathbb{A}}^{-s}_{x,t}$.

% Since
% \[
% \sigma a_q^k=\sigma A_{q+1}-E_q
% \]
% and $h_{q+1}=1$ on the support of $E_q$, the zero Fourier mode of $\rho_{q+1}^k$ cancels $E_q$, up to the harmless spatial constant $\sigma A_{q+1}h_{q+1}^k$. The nonzero modes of $\rho_{q+1}^k$ oscillate at frequencies $\gtrsim\mu_{q+1}$, and therefore small in $\dot{\mathbb A}^{-s}$ once $s$ is chosen large and $\lambda_{q+1}$ is chosen sufficiently large.

The remaining errors are perturbative. The Nash error contains lower powers of $w_{q+1}$, and is small because the intermittent slab is small in every $L^p$ with $p<k$. The dispersion error is controlled in the negative weighted Wiener norm by the $C_t^0L_x^1$ norm of $w_{q+1}$, which is small. The temporal error is controlled by taking a mean-free antiderivative of $w_{q+1}$, which gains a factor of frequency, and we design $w_{q+1}$ such that the loss from the time derivative is controlled by the $C_t^0L_x^1$ norm.

The induction below makes these estimates quantitative. It also keeps track of the absolute Fourier summability of the nonlinear products, which is needed in order to pass to the limit and obtain a weak singular solution in the sense of Definition~\ref{def:weak_para_soln}.

\section{Inductive Proposition}\label{sec:4}

\begin{proposition}\label{prop:ind}
    Fix $0 < \beta < \frac{1}{16}\left(1 - \frac{1}{k}\right)$. There are sequences $u_q:[0,1] \times \T \to \R$, $E_q : [0,1] \times \T \to \R$, and $\lambda_q \in \R$, $q \geq 0$, such that the following hold:
    \begin{enumerate}
        \item\label{i:1} Each $\lambda_q \geq 1$ is a power of $2$, $\lambda_{q'-1} < \lambda_{q'}$ for all $q' \leq q$ ($\lambda_{-1}=0$), and
        $$
        \sum_{q' \leq q}\lambda_{q'}^{-\beta} \leq \frac{1}{4} - 2^{-q-100}.
        $$
        \item\label{i:2} $(u_q,E_q)$ is a smooth solution to the relaxed equation given by ~\eqref{eq:relaxed}. In addition, $u_q$ has zero mean and the temporal support of both $u_q$ and $E_q$ is contained in
        $$
        \left(\frac{1}{2}- \sum_{q'\leq q} \frac{1}{\lambda_{q'}^\beta},\frac{1}{2} + \sum_{q' \leq q} \frac{1}{\lambda_{q'}^\beta}\right);
        $$
        \item\label{i:3} Upon setting $u_{-1}=0$, for every $q' \leq q$, the frequency support of $u_{q'} - u_{q'-1}$ lies in the annulus
        $$
        \left\{\xi : 2\lambda_{q'-1}^{3/2} < |\xi| < \lambda_{q'}^{3/2}\right\}.
        $$
        Also for every $1\le p<k$ and every
$q'\le q$ we have
\[
    \|u_{q'}-u_{q'-1}\|_{C_t^0L_x^p}
    \lesssim_p
    2^{-C(p)q'}.
\]
When $k\ge3$, for every $0<\alpha<\frac12-\frac1k$ and every $q'\le q$ we also
have
\[
    \|u_{q'}-u_{q'-1}\|_{C_t^0\dot H_x^\alpha}
    \lesssim_\alpha
    2^{-C(\alpha)q'}.
\]
In addition,
\[
    \|u_q\|_{C_t^0L_x^1}>\delta(1+2^{-q})
\]
for some $\delta>0$ chosen independent of $q$.
\item\label{i:4} The error $E_q$ satisfies the estimate
        \[
        \norm{E_q}_{\dot{\mathbb{A}}_{x,t}^{-s}} < 2^{-q-200}.
        \]
        \item\label{i:5} There is $C>0$ such that
        $$
        \sum_{\xi_1 + \ldots + \xi_k \not = 0} \left\Vert  \hat{u}_{q}(t,\xi_1) \hat{u}_{q}(t,\xi_2) \cdots \hat{u}_{q}(t,\xi_k) \right\Vert_{C_t^0}|\xi_1 + \ldots + \xi_k|^{-s} < C - 2^{-q-100}.
        $$
    \end{enumerate}
\end{proposition}

\begin{proof}[Proof of Theorem~\ref{thm:main} using Proposition~\ref{prop:ind}] We start by verifying the base case of the inductive hypothesis. We put $\lambda_0$ as a sufficiently large power of $2$ depending on $\beta$ so that $\lambda_0^{-\beta} \leq \frac{1}{8}$, and let
$u_0(t,x) = A \cos(2\pi \lambda_0 x)\chi(t)$ for some $A>0$ chosen later, $\chi \in C^\infty([0,1])$ not identically $0$ and supported in $\left(\frac{1}{2} - \frac{1}{ \lambda_0 ^\beta},\frac{1}{2}+\frac{1}{ \lambda_0 ^\beta}\right)$, and $E_0$ given by
$$
E_0(t,x) = \frac{A}{ 2\lambda_0 \pi} \sin(2 \lambda_0 \pi x) \chi'(t) - A(2 \lambda_0 \pi)^2 \cos(2 \lambda_0 \pi x)\chi(t) + \sigma\left(A \cos(2 \lambda_0 \pi x) \chi(t)\right)^k.
$$
Items \ref{i:1} and \ref{i:2} are obvious. Items \ref{i:3} and \ref{i:4} hold upon choosing $A > 0$ small enough and $\delta = 2^{-2}\Vert u_0 \Vert_{C_t^0L_x^1}$. Item \ref{i:5} holds for a large enough choice of $C$, and thus the base case holds.

Now we prove Theorem \ref{thm:main}. Let us assume items \ref{i:1}-\ref{i:5} hold for all $q$. Item \ref{i:3} implies that $u_q$ is a Cauchy sequence in $C_t^0 L^{p}_x \cap C_t^0 \dot{H}^{\alpha}_x$ which is a Banach space. Hence there is $u \in C_t^0 L^{p}_x \cap C_t^0 \dot{H}^{\alpha}_x$
such that $u_q \to u$. Since this holds for all $p < k$ and all $\alpha < \frac{1}{2} - \frac{1}{k}$ we conclude that
$$
u \in \bigcap_{\epsilon > 0} C_t^0 L^{k - \epsilon}_x \cap C_t^0 \dot{H}^{\frac{1}{2}-\frac{1}{k} - \epsilon}.
$$
In addition, item \ref{i:3} gives that $u$ is not identically $0$. From items \ref{i:1} and \ref{i:2}, the temporal support of $u$ is contained within $[1/4,3/4]$.

Now we verify that $u$ solves the \textsf{gKdV} equation in the sense of Definition \ref{def:weak_para_soln}. From items \ref{i:1} and \ref{i:2}, upon multiplying by a test function $\phi$ and integrating by parts we have that
\begin{equation}\label{eq:weak_form_relax}
    \int_0^1 \int_{\T} \left(u_q\partial_t\phi + u_q \partial_{xxx}\phi + \sigma u_q^k \partial_x\phi\right)\, dx \, dt = \int_0^1 \int_{\T} E_q \partial_x\phi\, dx\, dt.
\end{equation}

We will analyze \eqref{eq:weak_form_relax} term by term. So let us start with the first term on the left hand side. Using that $u_q \to u$ in $C_t^0L^1_x$, we see
\begin{equation}\label{eq:time_deriv_term_conv}
    \int_0^1 \int_{\T} u_q \partial_t\phi\, dx\, dt \to \int_0^1 \int_{\T} u \partial_t\phi\, dx\, dt = \int_0^1 \left\langle u,\partial_t\phi \right\rangle_{H^{-s},H^s}\, dt.
\end{equation}
Using precisely the same reasoning we have that
\begin{equation}\label{eq:triple_deriv_term_conv}
    \int_0^1 \int_{\T} u_q \partial_{xxx}\phi\, dx\, dt \to \int_0^1 \int_{\T} u \partial_{xxx}\phi\, dx\, dt = \int_0^1 \left\langle u,\partial_{xxx}\phi \right\rangle_{H^{-s},H^s}\, dt.
\end{equation}
Now analyzing the remaining term on the left hand side of ~\eqref{eq:weak_form_relax}, first using item \ref{i:5} we see that upon expanding
$$
\mathbb{P}_{\not=0}(u_q^k) = \sum_{\xi_1 + \ldots + \xi_k \not=0} \hat{u}_q(t,\xi_1) \cdots \hat{u}_q(t,\xi_k) e^{2\pi i (\xi_1 + \ldots + \xi_k)x}
$$
the right hand side is an absolutely summable series in $C_t^0 \dot{H}_x^{-s}$ which is uniformly bounded in $q$. Hence each $\mathbb{P}_{\not=0}(u_q^k)$ is a well-defined element of $C_t^0 \dot{H}_x^{-s}$ in the sense of Definition \ref{def:paras}. Now since $u_q \to u$ in $C_t^0 L_x^1$, it follows that for fixed $\xi$ we have $\hat{u}_q(t,\xi) \to \hat{u}(t,\xi)$ in $C_{t}^0$. Thus it follows that
$$
\lim_{q \to \infty} \left\Vert \hat{u}_q(t,\xi_1) \hat{u}_q(t,\xi_2) \cdots \hat{u}_q(t,\xi_k)\right\Vert_{C^0_t} = \left\Vert \hat{u}(t,\xi_1) \hat{u}(t,\xi_2) \cdots \hat{u}(t,\xi_k)\right\Vert_{C^0_t}.
$$
Applying the Fatou lemma we have
$$
\sum_{\xi_1 + \ldots + \xi_k \not = 0} \left\Vert \hat{u}(t,\xi_1) \hat{u}(t,\xi_2) \cdots \hat{u}(t,\xi_k)\right\Vert_{C^0_t} |\xi_1 + \xi_2 + \ldots +\xi_k|^{-s} \lesssim \liminf_{q \to \infty}\left(C - 2^{-q-100}\right) = C.
$$
Thus $\mathbb{P}_{\not=0}(u^k)$ is a well defined element of $C_t^0\dot{H}^{-s}_x$. Finally we have
\begin{equation*}
    \begin{split}
        \Vert \mathbb{P}_{\not=0}(u^k) - \mathbb{P}_{\not=0}(u_q^k) \Vert_{C_t^0 \dot{H}_x^{-s}} &\lesssim \sum_{\xi_1 + \ldots + \xi_k \not=0} \Vert \hat{u}_q(t,\xi_1) \cdots \hat{u}_q(t,\xi_k) - \hat{u}(t,\xi_1)\cdots \hat{u}(t,\xi_k)\Vert_{C^0_t} |\xi_1 + \ldots + \xi_k|^{-s}.
    \end{split}
\end{equation*}
By the frequency support assertion in item~\ref{i:3}, $u_q$ is the truncation of $u$ to the frequency blocks up to scale $\lambda_q^{3/2}$. Hence the summand above vanishes whenever
\[
\max_i |\xi_i|\le \lambda_q^{3/2}.
\]
Therefore
$$
   \Vert \mathbb{P}_{\not=0}(u^k) - \mathbb{P}_{\not=0}(u_q^k) \Vert_{C_t^0 \dot{H}_x^{-s}} \lesssim \sum_{\substack{\xi_1+\ldots+\xi_k \not=0\\\max_i|\xi_i|>\lambda_q^{3/2}}} \left\|\hat u(t,\xi_1) \cdots \hat u(t,\xi_k)\right\|_{C_t^0}|\xi_1 +\cdots \xi_k|^{-s}
$$
The right-hand side is the tail of the absolutely convergent series defining $\mathbb P_{\neq0}(u^k)$, and hence tends to zero upon sending $q \to \infty$. Thus
$$
\lim_{q \to \infty} \mathbb{P}_{\not=0}(u_q^k) = \mathbb{P}_{\not=0}(u^k).
$$
Now using the uniform convergence in the time variable we have
\begin{equation}\label{eq:nonlin_term_conv}
    \begin{split}
        \int_0^1 \int_{\T} \sigma u_q^k \partial_{x}\phi\, dx\, dt &= \int_0^1 \int_{\T} \sigma \mathbb{P}_{\not=0}\left(u_q^k\right) \partial_{x}\phi\, dx\, dt\\
        &= \int_0^1 \sigma \left\langle \mathbb{P}_{\not=0}\left(u_q^k\right), \partial_{x}\phi\right\rangle_{\dot{H}^{-s},\dot{H}^s}\, dt\\
        &\to \int_0^1 \sigma \left\langle \mathbb{P}_{\not=0}\left(u^k\right), \partial_{x}\phi\right\rangle_{\dot{H}^{-s},\dot{H}^s}\, dt.
    \end{split}
\end{equation}
To analyze the final term involving the error, we utilize item \ref{i:4} and the Parseval theorem to see that
\begin{equation}\label{eq:error_term_conv}
\begin{split}
    \left|\int_0^1 \int_{\T} E_q \partial_x\phi\, dx\, dt\right| &= \left| \int_0^1 \sum_{\xi \not =0} \hat{E}_q(t,\xi) \overline{\widehat{\partial_x\phi}}(t,\xi)\, dt\right|\\
    &\lesssim \sum_{\xi\not=0} \left\Vert \hat{E}_q(t,\xi) \right\Vert_{C_t^0} \left\Vert \widehat{\partial_x\phi}(t,\xi)\right\Vert_{C_t^0}\\
    &\lesssim \left(\sup_{\xi \not=0} |\xi|^{s+1} \left\Vert \hat{\phi}(t,\xi)\right\Vert_{C_t^0}\right) \Vert E_q \Vert_{\dot{\mathbb{A}}^{-s}_{t,x}}\\
    &\lesssim 2^{-q-200}\\
    &\to 0
    \end{split}
\end{equation}
In ~\eqref{eq:error_term_conv} notice we have utilized that $\phi$ is smooth to conclude that
$$
\sup_{\xi \not=0} |\xi|^{s+1} \left\Vert \hat{\phi}(t,\xi)\right\Vert_{C_t^0} \lesssim 1.
$$
Thus upon sending $q \to \infty$ in ~\eqref{eq:weak_form_relax} and utilizing ~\eqref{eq:time_deriv_term_conv}, ~\eqref{eq:triple_deriv_term_conv}, ~\eqref{eq:nonlin_term_conv}, and ~\eqref{eq:error_term_conv} we have that
\begin{equation}\label{eq:solution}
    \int_0^1 \left(\left\langle u,\partial_t\phi \right\rangle_{H^{-s},H^s} + \left\langle u,\partial_{xxx}\phi \right\rangle_{H^{-s},H^s} + \sigma \left\langle \mathbb{P}_{\not=0}\left(u^k\right), \partial_{x}\phi\right\rangle_{\dot{H}^{-s},\dot{H}^s}\right)\, dt = 0.
\end{equation}
From ~\eqref{eq:solution} we see $u$ is our desired solution in the sense of Definition \ref{def:weak_para_soln}.

\end{proof}

\section{Proof of Inductive Proposition}\label{sec:5}

\subsection{Parameters}\label{subsec:param} Throughout the proceeding construction we will work with many different parameters, and we summarize them here. First we will take $\epsilon_q = 2^{-q}$ and $\gamma_q = 1 + \epsilon_q$. These parameters are fixed at every level of the iteration. Now assuming parameters at level $q$ are fixed, we will choose a large integer $A_{q+1}$. This $A_{q+1}$ will at a minimum satisfy the estimate
$$
A_{q+1} > 2\Vert E_q \Vert_{C_t^0 L_x^\infty}
$$
but it will be chosen to satisfy \eqref{eq:A_q+1_cond} and \eqref{eq:A_q+1_cond2}. $A_{q+1}$ will play no role before the proof of item \ref{i:5}, so for the proofs of items \ref{i:1}-\ref{i:4} one may regard $A_{q+1} \simeq 1$. Next we define a high frequency parameter $\lambda_{q+1}$ which will be an extremely large power $2$ such that $\lambda_{q+1}^{\epsilon_{q+1}/2}$ is also an integer. $\lambda_{q+1}$ is allowed to depend on the parameter $A_{q+1}$, and in particular, its size is determined by satisfying \eqref{eq:support_cond}, \eqref{eq:lambda-Lp-small}, \eqref{eq:sobolev_est_q_gtr_q_alpha}, \eqref{eq:lambda-Lone-small}, \eqref{eq:ED-decay}, \eqref{eq:EN-decay}, \eqref{eq:wiener_est_nonlinearity}, \eqref{eq:EO-decay}, \eqref{eq:ET-decay}, \eqref{eq:n=0_case}, \eqref{eq:n=1_case}, \eqref{eq:AF-case3-gap}, \eqref{eq:n=1_case_2}, \eqref{eq:page_29_est}, \eqref{eq:page29_est}, \eqref{eq:page_29_est_3}, \eqref{eq:page_32_conv_est}, \eqref{eq:mixed-prod-bound}, and \eqref{eq:pure-prod-small}. For convenience, we put $\mu_{q+1} = \lambda_{q+1}^{\epsilon_{q+1}}$ and $\nu_{q+1} = \lambda_{q+1}^{\gamma_{q+1}}$. So we have the relation $\nu_{q+1} = \lambda_{q+1}\mu_{q+1}$.

\subsection{Construction of Increment}

\begin{definition}[\textbf{Increment}]\label{def:wq+1}
    Put $\rho_{q+1} = \rho_{k,\lambda_{q+1},\epsilon_{q+1}}$ from Lemma \ref{lem:boldW} and let
    \begin{equation}\label{eq:wq+1}
        w_{q+1}(t, x) = \mathbb{P}_{\leq \mu_{q+1}^{1/2}} \left(a_{q}(t,x) \right) \mathbb{P}_{\leq \nu_{q+1}}\left(\rho_{q+1}(x)\right) h_{q+1}(t)
    \end{equation}
    where we put
    \begin{equation}\label{eq:ak}
        a_{q}(t,x) = \left(A_{q+1} - \sigma E_q(t,x)\right)^{1/k},
    \end{equation}
    and $h_{q+1}$ is a smooth function which is identically $1$ on
    $$
    \left(\frac{1}{2}- \sum_{q'\leq q} \frac{1}{\lambda_{q'}^\beta},\frac{1}{2} + \sum_{q' \leq q} \frac{1}{\lambda_{q'}^\beta}\right),
    $$
    $0$ outside
    $$
    \left(\frac{1}{2}- \sum_{q'\leq q+1} \frac{1}{\lambda_{q'}^\beta},\frac{1}{2} + \sum_{q' \leq q+1} \frac{1}{\lambda_{q'}^\beta}\right),
    $$
    and satisfies the estimates $0 \leq h_{q+1}(t) \leq 1$ and  $|h_{q+1}'(t)| \lesssim \lambda_{q+1}^\beta$.
\end{definition}

Throughout we will assume $u_q$ is a fixed function satisfying all of the hypotheses in Proposition \ref{prop:ind} and we define $u_{q+1}$ by~\eqref{eq:uq+1}. Our goal will be to show with these choices, $u_{q+1}$ satisfies all of the conditions in Proposition \ref{prop:ind}.

\begin{lemma}[\textbf{$C_t^0 L_x^p$ estimate for $w_{q+1}$}]\label{lem:wq+1_est}
    For $1 \leq p \leq \infty$, one has 
    \begin{equation*}
        \norm{w_{q+1}}_{C_t^0 L_x^p} \lesssim \lambda_{q+1}^{(1 - \epsilon_{q+1}) (\frac{1}{k} - \frac{1}{p})}
    \end{equation*}
\end{lemma}
\begin{proof}
    We compute for any $t \in [0, 1]$ and use Lemma~\ref{lem:boldW}.
    \[\norm{w_{q+1}(t, \cdot)}_{ L_x^p}^p \lesssim \norm{\rho_{q+1}}_{L^p}^p |h_{q+1}(t)|^p \lesssim \norm{\rho_{q+1}}_{L^p}^p. \]
    Now taking the supremum in $t$ gives the result.
\end{proof}

\subsection{Proof of Item \ref{i:1}}
Let us suppose that $\lambda_0 < \lambda_1 < \ldots < \lambda_q$ are all chosen to be powers of $2$ and satisfy
$$
\sum_{q' \leq q} \frac{1}{\lambda_{q'}^\beta} \leq \frac{1}{4} - 2^{-q-100}.
$$
$\lambda_{q+1}$ is chosen to be a power of $2$ which satisfies the requirements in \ref{subsec:param} as well as
\begin{equation*}
    \lambda_{q+1} > \max\left(\lambda_q,2^{\frac{q+101}{\beta}}\right).
\end{equation*}
Then we see
$$
\sum_{q' \leq q+1} \frac{1}{\lambda_{q'}^\beta} = \sum_{q'\leq q} \frac{1}{\lambda_{q'}^\beta} + \frac{1}{\lambda_{q+1}^\beta} \leq \frac{1}{4} - 2^{-q-100} + 2^{-q-101} = \frac{1}{4} - 2^{-q-101}.
$$

\subsection{Proof of Item \ref{i:2}}

We argue by induction on $q$. The case $q=0$ is immediate from the initial construction. Assume item \ref{i:2} holds at stage $q$. We verify it for stage $q+1$.

Smoothness of $u_{q+1}$ follows directly from Definition~\ref{def:wq+1}: by the inductive hypothesis $E_q$ is smooth, and since $A_{q+1}$ is chosen so that $A_{q+1}-\sigma E_q(t,x)>0$ everywhere, the amplitude $a_q=(A_{q+1}- \sigma E_q)^{1/k}$ is smooth. The spatial Fourier cutoffs preserve smoothness; $\rho_{q+1}$ is smooth by Lemma~\ref{lem:boldW}; and $h_{q+1}$ is smooth by construction. Thus $w_{q+1}$, and therefore $u_{q+1}=u_q+w_{q+1}$, are smooth. Since $E_{q+1}$ is defined as the sum $E_O+E_N+E_D+E_T$, each term being obtained from $u_q$, $E_q$, and $w_{q+1}$ by multiplication and differentiation, $E_{q+1}$ is smooth as well.

We next prove that $w_{q+1}$ has zero spatial mean. By Definition~\ref{def:wq+1},
\[
w_{q+1} = \mathbb P_{\le \mu_{q+1}^{1/2}}(a_q(t,x))\mathbb P_{\le \nu_{q+1}}(\rho_{q+1}(x)) h_{q+1}(t).
\]
From Definition \ref{def:projs}, the first factor has spatial Fourier support in $\{|\xi|\le 2\mu_{q+1}^{1/2}\}$. By Lemma~\ref{lem:fourier}, the Fourier support of $\rho_{q+1}$ lies in $\lambda_{q+1}^{\epsilon_{q+1}}\mathbb Z$, and since $\rho_{q+1}$ has zero mean by Lemma~\ref{lem:boldW}, its zero Fourier mode vanishes. Therefore every nonzero Fourier mode of $\mathbb P_{\le \nu_{q+1}}(\rho_{q+1})$ has size at least $\lambda_{q+1}^{\epsilon_{q+1}}$. If the product had a zero Fourier mode, there would exist frequencies $\ell,m$ with
\[
|\ell|\le 2\lambda_{q+1}^{\frac{\epsilon_{q+1}}{2}},\qquad m\in \lambda_{q+1}^{\epsilon_{q+1}}\mathbb Z\setminus\{0\},\qquad \ell+m=0.
\]
But this is impossible, since $|m|\ge \lambda_{q+1}^{\epsilon_{q+1}}>2\lambda_{q+1}^{\frac{\epsilon_{q+1}}{2}}\ge |\ell|$. Hence
\[
\hat{w}_{q+1}(t,0)=0.
\]for every $t$. Since $u_q$ has zero mean by the inductive hypothesis, it follows that
\[
\hat{u}_{q+1}(t,0)=\hat{u}_q(t,0)+\hat{w}_{q+1}(t,0)=0.
\]
Thus $u_{q+1}$ has zero mean.

Now that $w_{q+1}$ is known to have zero mean, the mean-free antiderivative $\omega_{q+1}$ appearing in the definition of $E_T$ is well-defined.

We now verify the relaxed equation. Since $(u_q,E_q)$ solves \eqref{eq:relaxed}, we have
\[
\partial_t u_q+\partial_{xxx}u_q+\sigma\partial_x(u_q^k)=\partial_x E_q.
\]
Using $u_{q+1}=u_q+w_{q+1}$, we compute
\[
\partial_t u_{q+1}+\partial_{xxx}u_{q+1}+\sigma\partial_x(u_{q+1}^k) = \partial_x E_q+\partial_t w_{q+1}+\partial_{xxx}w_{q+1} +\sigma\partial_x\bigl((u_q+w_{q+1})^k-u_q^k\bigr).
\]
Expanding the nonlinear term by the binomial theorem and using the definitions of $E_O$, $E_N$, $E_D$, and $E_T$ in \eqref{eq:Rq+1}, together with $\partial_x\omega_{q+1}=w_{q+1}$, gives
\[
\partial_t u_{q+1}+\partial_{xxx}u_{q+1}+\sigma\partial_x(u_{q+1}^k)
= \partial_x(E_O+E_N+E_D+E_T) = \partial_x E_{q+1}.
\]
Hence $(u_{q+1},E_{q+1})$ solves the relaxed equation \eqref{eq:relaxed}.

It remains to verify the temporal support. By the inductive hypothesis,
\[
\operatorname{supp}_t u_q,\ \operatorname{supp}_t E_q \subset I_q :=
\left(\frac12-\sum_{q'\le q}\frac1{\lambda_{q'}^\beta}, \frac12+\sum_{q'\le q}\frac1{\lambda_{q'}^\beta}\right).
\]
By construction,
\[
\operatorname{supp}_t h_{q+1}\subset
I_{q+1} := \left(\frac12-\sum_{q'\le q+1}\frac1{\lambda_{q'}^\beta},\, \frac12+\sum_{q'\le q+1}\frac1{\lambda_{q'}^\beta}\right).
\]
Since $h_{q+1}$ is a factor in $w_{q+1}$, it follows immediately from Definition~\ref{def:wq+1} that
\[
\operatorname{supp}_t w_{q+1}\subset I_{q+1}.
\]
Therefore
\[
\operatorname{supp}_t u_{q+1}\subset I_{q+1}.
\]
Likewise, each of $E_O$, $E_N$, $E_D$, and $E_T$ is obtained from $u_q$, $E_q$, $w_{q+1}$ by multiplication, $x$-differentiation, $x$-anti-differentiation, or $t$-differentiation; none of these enlarges temporal support. Since $u_q$ and $E_q$ are supported in $I_q\subset I_{q+1}$ and $w_{q+1}$ is supported in $I_{q+1}$, each of $E_O$, $E_N$, $E_D$, and $E_T$ is supported in $I_{q+1}$. Hence
\[
\operatorname{supp}_t E_{q+1}\subset I_{q+1}.
\]

This proves item \ref{i:2}.

\subsection{Proof of Item \ref{i:3}}

We again argue by induction on $q$. The case $q=0$ was arranged in the initial construction. Assume item \ref{i:3} holds at stage $q$. We verify all assertions for stage $q+1$.

First from~\eqref{eq:wq+1} and Definition \ref{def:projs} we may see that the frequency support of $w_{q+1}$ lies in
$$
\left\{\xi : \mu_{q+1} - 2\mu_{q+1}^{1/2} \leq |\xi| \leq 2\left(\nu_{q+1} + \mu_{q+1}^{1/2}\right)\right\}.
$$
Upon choosing $\lambda_{q+1}$ large enough we may ensure that
\begin{equation}\label{eq:support_cond}
    2\lambda_q^{3/2} < \mu_{q+1} - 2\mu_{q+1}^{1/2} \quad \text{and} \quad 2\left(\nu_{q+1} + \mu_{q+1}^{1/2}\right) < \lambda_{q+1}^{3/2}
\end{equation}
since recall $\epsilon_{q+1} = 2^{-(q+1)}$, $\mu_{q+1} = \lambda_{q+1}^{\epsilon_{q+1}}$, and $\nu_{q+1} = \lambda_{q+1}^{1 + \epsilon_{q+1}}$. This establishes the first claim.

Recall from Lemma \ref{lem:wq+1_est} that
$$
\Vert w_{q+1} \Vert_{C_t^0 L^p_x} \lesssim \lambda_{q+1}^{(\epsilon_{q+1} - 1)\left(\frac{1}{p} - \frac{1}{k}\right)}.
$$
Since $p<k$, the exponent is negative. Hence, choosing $\lambda_{q+1}$ sufficiently large, we may ensure
\begin{equation}\label{eq:lambda-Lp-small}
    \|w_{q+1}\|_{C_t^0L_x^p}
    \lesssim_p
    2^{-C(p)(q+1)}
\end{equation}
for some $C(p) >0$ depending only on $p$. We next prove the Sobolev increment estimate when $k \ge 3$. Fix 
\[
    0<\alpha<\frac12-\frac1k.
\]
and set
\[
s_k := \frac12-\frac1k, \qquad \alpha_\ast:= \frac{\alpha +s_k}{2}.
\]
Thus
\[
\alpha < \alpha_\ast <s_k.
\]
Put
$$
q(\alpha) = \max \left\{q : \frac{1-\epsilon_q}{1 + \epsilon_q}s_k < \alpha_\ast \right\}.
$$
Recall we had put $\epsilon_q = 2^{-q}$, and so
$$
\frac{1 - \epsilon_q}{1 + \epsilon_q}s_k \to s_k
$$
in a monotonically increasing manner. Since $\alpha_* < \frac{1}{2} - \frac{1}{k}$, $q(\alpha)$ exists and is finite. Now we consider two cases. First when $q+1 \leq q(\alpha)$, then we have
\begin{equation}\label{eq:q_less_q_alpha}
    \Vert w_{q+1} \Vert_{C_t^0 \dot{H}_x^\alpha} < \left(2^{q(\alpha)+100} \max_{1 \leq q \leq q(\alpha)} \Vert w_q \Vert_{C_t^0 \dot{H}_x^\alpha}\right) 2^{-(q+1)}.
\end{equation}
When $q+1 > q(\alpha)$ then using a variant of Bernstein's inequality\footnote{A variant of the argument used in \cite[Lemma 5.8]{GR} suffices.} and the observation above that the maximum frequency of $w_{q+1}$ is $2(\nu_{q+1} + \mu_{q+1}^{1/2}) \simeq \nu_{q+1}$ we have
\begin{equation}\label{eq:q_gtr_q_alpha}
    \Vert w_{q+1} \Vert_{C_t^0 \dot{H}^\alpha_x} \lesssim \lambda_{q+1}^{(1 + \epsilon_{q+1})\alpha} \lambda_{q+1}^{-(1 - \epsilon_{q+1})s_k} \lesssim \lambda_{q+1}^{-\frac12(s_k -\alpha)}.
\end{equation}
Choosing
\begin{equation}\label{eq:sobolev_est_q_gtr_q_alpha}
    \lambda_{q+1} > 4^{q+1}
\end{equation}
and combining \eqref{eq:q_less_q_alpha} and \eqref{eq:q_gtr_q_alpha} we ensure
\begin{equation*}
    \|w_{q+1}\|_{C^0_t\dot{H}^\alpha_x} < \left(2^{q(\alpha)+100} \max_{1 \leq q \leq q(\alpha)} \Vert w_q \Vert_{C_t^0 \dot{H}_x^\alpha}\right) 2^{-(q+1)} + 2^{-(s_k-\alpha)(q+1)} \lesssim 2^{-(s_k-\alpha)(q+1)}
\end{equation*}
which establishes the second claim.

It remains to prove the $C_t^0L^1_x$ lower bound. From Lemma \ref{lem:wq+1_est} we have
\[
\|w_{q+1}\|_{C_t^0L_x^1}
\lesssim
\lambda_{q+1}^{(1-\epsilon_{q+1})\left(\frac1k-1\right)}.
\]
Since
\[
(1-\epsilon_{q+1})\left(\frac1k-1\right)<0,
\]
choosing $\lambda_{q+1}$ still larger if necessary, we ensure
\begin{equation}\label{eq:lambda-Lone-small}
    \|w_{q+1}\|_{C_t^0L_x^1}
    <
    \delta 2^{-q-1}.
\end{equation}
Using the reverse triangle inequality and the inductive hypothesis we have
\[
\begin{aligned}
\|u_{q+1}\|_{C_t^0L_x^1}
&\ge
\|u_q\|_{C_t^0L_x^1}
-
\|w_{q+1}\|_{C_t^0L_x^1} \\
&>
\delta(1+2^{-q})
-
\delta 2^{-q-1}\\
&=
\delta(1+2^{-q-1}).
\end{aligned}
\]
This is exactly the third inequality in item~\ref{i:3} at stage $q+1$.

\subsection{Proof of Item \ref{i:4}}

We argue by induction on $q$. $q = 0$ follows from initial construction. Assume item~\ref{i:4} holds at stage $q$, we verify for $q +1$. Recall $E_{q + 1} = E_O + E_N + E_D + E_T$, where these quantities were defined in \eqref{eq:Rq+1}, \eqref{eq:osc_error}, \eqref{eq:Nash_error}, \eqref{eq:dis_error}, \eqref{eq:temporal_error}.

\subsubsection*{Dispersion Error}
We compute
\begin{equation*}
    \norm{E_D}_{\dot{\mathbb{A}}_{x,t}^{-s}} = \sum_{\xi \neq 0}|\xi|^{-s} \norm{\hat{E}_D(\cdot,\xi)}_{C_t^0} \lesssim \sum_{\xi \neq 0} |\xi|^{2 - s} \norm{\hat{w}_{q+1}(\cdot, \xi)}_{C_t^0}
\end{equation*}
% In view of Definition \ref{def:wq+1} and the smooth frequency cutoff defined in Definition \ref{def:projs}, the first factor has frequency support in 
% \[\{|\ell| \leq 2\lambda_{q+1}^{\epsilon_{q+1}/2}\}\] On the other hand every nonzero Fourier mode of $\rho_{q+1}$ lies in 
% $\lambda_{q+1}^{\epsilon_{q+1}} \Z \setminus \{0\}$ according to Lemma~\ref{lem:fourier}. Using Lemma~\ref{lem:boldW} we know that $\int_\T \rho_{q+1} = 0$. Thus every nonzero mode $m$ of the second factor satisfies 
% \[|m| \geq \lambda_{q+1}^{\epsilon_{q+1}}\]
% Thus for $\xi = \ell  +m$ as Fourier mode of $w_{q+1}$
% \[|\xi| \geq |m| - |\ell| \geq \lambda_{q+1}^{\epsilon_{q+1}} - 2\lambda_{q+1}^{\epsilon_{q+1}/2}\] so for $\lambda_{q+1}$ sufficiently large, one has 
% \begin{equation}\label{eq:w-support-gap2}
%     \operatorname{supp}\hat{w}_{q+1}(t,\cdot)\subset \{|\xi|\gtrsim \lambda_{q+1}^{\epsilon_{q+1}}\}.
% \end{equation}
Combining this with 
$$
\|\hat{w}_{q+1}(\cdot,\xi)\|_{C_t^0}\le \|w_{q+1}\|_{C_t^0L_x^1},
$$
we obtain
\begin{align*}
   \norm{E_D}_{\dot{\mathbb{A}}_{x,t}^{-s}}
&\lesssim
\|w_{q+1}\|_{C_t^0L_x^1}
\sum_{\xi \not=0} |\xi|^{2-s}.
\end{align*}
For fixed $s >3$ the series converges and satisfies the estimate
\begin{align*}
  \norm{E_D}_{\dot{\mathbb{A}}_{x,t}^{-s}}
&\lesssim\|w_{q+1}\|_{C_t^0L_x^1}
\end{align*}
Using Lemma~\ref{lem:wq+1_est} we have
\[\|w_{q+1}\|_{C_t^0L_x^1} \lesssim \lambda_{q+1}^{(1 - \epsilon_{q+1})(\frac{1}{k} - 1)} \] 

Thus 
\begin{equation}\label{eq:ED-decay}
   \norm{E_D}_{\dot{\mathbb{A}}_{x,t}^{-s}}
    < 2^{-(q+1) - 200 - 2}
\end{equation}
upon choosing $\lambda_{q+1}$ sufficiently large.

\subsubsection*{Nash Error}
Recall 
\[  E_N = \sigma \sum_{n=1}^{k-1} \binom{k}{n} w_{q+1}^n u_q^{k-n}\]
so we estimate each summand in $\dot{\mathbb{A}}_{x,t}^{-s}$. 

Since for $s > 1$
\[
\begin{aligned}
    \|w_{q+1}^n u_q^{k-n}\|_{\dot{\mathbb A}_{x,t}^{-s}}
    &=
    \sum_{\xi\neq0} |\xi|^{-s} \left\|(w_{q+1}^n u_q^{k-n})^{\wedge}(\cdot,\xi)\right\|_{C_t^0} \\
    &\le
    \sum_{\xi\neq0} |\xi|^{-s} \|w_{q+1}^n u_q^{k-n}\|_{C_t^0L_x^1} \\
    &\lesssim_s
    \|w_{q+1}^n u_q^{k-n}\|_{C_t^0L_x^1}.
\end{aligned}
\]
it suffices to estimate the $L^1$-norm.

Now, for each $1 \leq n \leq k - 1$
\[
    \|w_{q+1}^n u_q^{k-n}\|_{C_t^0L_x^1}
    \le
    \|u_q\|_{C_t^0L_x^\infty}^{k-n}
    \|w_{q+1}\|_{C_t^0L_x^n}^n.
\]
Since $u_q$ is fixed at stage $q$, Lemma~\ref{lem:wq+1_est} gives 
\[
\|w_{q+1}^n u_q^{k-n}\|_{C_t^0L_x^1}
\lesssim_q
\|w_{q+1}\|_{C_t^0L_x^n}^n
\lesssim_q
\lambda_{q+1}^{(1-\epsilon_{q+1})\frac{n-k}{k}}.
\]
Summing over $n=1,\ldots,k-1$, we obtain
\[
\|E_N\|_{\dot{\mathbb A}_{x,t}^{-s}}
\lesssim_{q}
\lambda_{q+1}^{-\frac{1-\epsilon_{q+1}}{k}}.
\]
Thus by choosing $\lambda_{q+1}$ sufficiently large, the Nash error satisfies
\begin{equation}\label{eq:EN-decay}
    \|E_N\|_{\dot{\mathbb A}_{x,t}^{-s}} < 2^{-(q+1)-200-2}.
\end{equation}
\subsubsection*{Oscillation Error}

We start with expression of \eqref{def:wq+1}, and write
\begin{equation*}
    \begin{split}
        w_{q+1}^k &=\left( a_q(t,x)-\mathbb{P}_{>\mu_{q+1}^{1/2}}\left( a_q(t,x)\right)  \right)^k \left( \rho_{q+1}(x) -\mathbb{P}_{>\nu_{q+1}} (\rho_{q+1}(x)) \right)^k h_{q+1}^k(t)
        \\
        &= \left([a_q(t,x)]^k +Q_1(a_q,\lambda_{q+1}^{\epsilon_{q+1}/2}) \right) \left( \rho_{q+1}^k +Q_2(\rho_{q+1},\lambda^{1+\epsilon_{q+1}}_{q+1})\right)h_{q+1}^k(t)
    \end{split}
\end{equation*}
where 
\begin{equation*}
    Q_1(a_q,\lambda_{q+1}^{\epsilon_{q+1}/2})=\sum_{i=1}^k(-1)^i \binom{k}{i}  [a_q(t,x)]^{k-i} \left[ \mathbb{P}_{>\mu_{q+1}^{1/2}} (a_q(t, x))  \right]^i 
\end{equation*}

and 
\begin{equation*}
    Q_2(\rho_{q+1},\lambda_{q+1}^{1+\epsilon_{q+1}})=\sum_{i=1}^k(-1)^i \binom{k}{i}  [\rho_{q+1}(x)]^{k-i} \left[ \mathbb{P}_{>\nu_{q+1}} (\rho_{q+1}(x))  \right]^i .
\end{equation*}

We further rewrite 
\begin{equation*}
        w_{q+1}^k= a^k_q(t,x)\rho_{q+1}^k(x)h_{q+1}^k(t)+ \mathsf{R}_{q+1}
\end{equation*}
where 

\begin{equation*}
   \mathsf{R}_{q+1}= [\rho_{q+1}^k Q_1+a_q^kQ_2+Q_1Q_2 ]h_{q+1}^k.
\end{equation*}

Using $a_q(t,x):=(A_{q+1}-\sigma E_q(t,x))^{1/k}$ so that 
\[\sigma a_q^k = \sigma A_{q+1} - E_q\]
One may write 
\begin{align*}
    E_q + \sigma w_{q+1}^k &= E_q + (\sigma A_{q+1} - E_q)\rho_{q+1}^k(x) h_{q+1}^k(t) + \sigma \mathsf{R}_{q+1}
\end{align*}
Now decompose 
\begin{align*}
    \rho_{q+1}^k &= \P_{ = 0}(\rho_{q+1}^k) + \P_{\neq 0}(\rho_{q+1}^k) 
\end{align*}
Using the normalization as in Lemma~\ref{lem:boldW}
\[\P_{ = 0}(\rho_{q+1}^k) = \int_\T \rho_{q+1}^k dx = 1\]
and that $h_{q+1} \equiv 1$ on $\mathrm{supp}_t E_q$, one obtain 
\[E_q h_{q+1}^k = E_q\]
Therefore 
\begin{align*}
    E_q + (\sigma A_{q+1} - E_q)\rho_{q+1}^k(x) h_{q+1}^k(t) &= E_q + (\sigma A_{q+1} - E_q)(1 + \P_{\neq 0}(\rho_{q+1}^k)) h_{q+1}^k(t)\\
    &= \sigma A_{q+1} h_{q+1}^k + (\sigma A_{q+1} - E_q) \P_{\neq 0}(\rho_{q+1}^k) h_{q+1}^k
\end{align*}
The term $\sigma A_{q+1} h_{q+1}^k $ is independent of $x$ so it vanishes under the ${\dot{\mathbb{A}}_{x,t}^{-s}}$ norm.

Consequently 
\[\norm{E_O}_{\dot{\mathbb{A}}_{x,t}^{-s}} \leq \norm{(\sigma A_{q+1} - E_q) \P_{\neq 0}(\rho_{q+1}^k) h_{q+1}^k}_{\dot{\mathbb{A}}_{x,t}^{-s}} + \norm{\mathsf{R}_{q+1}}_{\dot{\mathbb{A}}_{x,t}^{-s}}\]

We first analyze the principal oscillatory term. Via Lemma~\ref{lem:pseudo_diff_cont_Wiener}
\begin{align*}
     \norm{(\sigma A_{q+1} - E_q) \P_{\neq 0}(\rho_{q+1}^k) h_{q+1}^k}_{\dot{\mathbb{A}}_{x,t}^{-s}} &\lesssim \norm{(\sigma A_{q+1} - E_q) h_{q+1}^k}_{\mathbb{A}_{x,t}^{s}} \norm{\P_{\neq 0}(\rho_{q+1}^k)}_{\dot{\A}^{-s}}
\end{align*}
Since $\sigma A_{q+1} - E_q$ depends only on $\lambda_{q+1}$-independent quantities, one obtains
\[ \norm{(\sigma A_{q+1} - E_q) h_{q+1}^k}_{\mathbb{A}_{x,t}^{s}}  \lesssim_q 1\]

Now in view of
 \eqref{eq:fourier_series rho k} we have
\begin{align*}
    \mathbb{P}_{ \neq 0}( \rho_{q+1}^k(x)) &= \sum_{n \neq 0} \hat{g}_k(\lambda_{q+1}^{\epsilon_{q+1} - 1} n) e^{2\pi i \lambda_{q+1}^{\epsilon_{q+1}} n x}
\end{align*}
for some $g_k$ as in Lemma~\ref{lem:fourier}. Now in the weighted Wiener norm one obtains
\begin{align*}
    \norm{\P_{\neq 0}(\rho_{q+1}^k(x))}_{\dot{\A}^{-s}} &= \sum_{n \neq 0} |\lambda_{q+1}^{\epsilon_{q+1}} n|^{-s} |\hat{g}_k (\lambda_{q+1}^{\epsilon_{q+1} -1} n)| = \lambda_{q+1}^{-s\epsilon_{q+1}} \sum_{n \neq 0} |n|^{-s} |\hat{g}_k (\lambda_{q+1}^{\epsilon_{q+1} -1} n)|
\end{align*}
Since $\hat{g}_k$ is Schwartz, then
\[|\hat{g}_k(\lambda_{q+1}^{\epsilon_{q+1} - 1} n)| \lesssim 1\]
Thus taking $s > 1$ so that $\sum_{n \neq 0} |n|^{-s}$ converges, 
we have
\begin{equation}\label{eq:wiener_est_nonlinearity}
    \norm{\P_{\neq 0}(\rho_{q+1}^k(x))}_{\dot{\A}^{-s}}\lesssim \lambda_{q+1}^{-s\epsilon_{q+1}} < 2^{-(q+1)-200 - 8}
\end{equation}
for any fixed $s > 1$, by choosing $\lambda_{q+1}$ sufficiently large.

Now it suffices to bound $\mathsf{R}_{q+1}$. Since the torus $\T$ has finite measure, for any $s > 1$
\begin{align*}
    \norm{\mathsf{R}_{q+1}}_{\dot{\mathbb{A}}_{x,t}^{-s}} &\lesssim  \norm{\mathsf{R}_{q+1}}_{C_t^0 L_x^\infty} \lesssim \norm{\rho_{q+1}^k Q_1}_{C_t^0 L_x^\infty}  + \norm{a_{q}^k Q_2}_{C_t^0 L_x^\infty} + \norm{Q_1 Q_2}_{C_t^0 L_x^\infty}
\end{align*}
It suffices to estimate the high-frequency tails in view of 
\begin{align*}
    \norm{Q_1}_{C_t^0 L_x^\infty} &\lesssim \sum_{i = 1}^k \norm{a_q}_{C_t^0 L_x^\infty}^{k - i} \norm{\P_{> \mu_{q+1}^{1/2}}(a_q)}_{C_t^0 L_x^\infty}^i \lesssim \norm{\P_{> \mu_{q+1}^{1/2}}(a_q)}_{C_t^0 L_x^\infty}\\
      \norm{Q_2}_{C_t^0 L_x^\infty} &\lesssim  \sum_{i = 1}^k \norm{\rho_{q+1}}_{ L_x^\infty}^{k - i} \norm{\P_{> \nu_{q+1}}(\rho_{q+1})}_{L_x^\infty}^i \lesssim \norm{\rho_{q+1}}_{ L_x^\infty}^{k - 1}\norm{\P_{> \nu_{q+1}}(\rho_{q+1})}_{L_x^\infty}\\
      &\lesssim \lambda_{q+1}^{(1 - \epsilon_{q+1}) \frac{k-1}{k}}\norm{\P_{> \nu_{q+1}}(\rho_{q+1})}_{L_x^\infty}
\end{align*}

For the tail frequencies, using Lemma \ref{lem:freq_proj_rho} for every integer $M \geq 1$
\begin{align*}
    \norm{\P_{> \mu_{q+1}^{1/2}}(a_q)}_{C_t^0 L_x^\infty} &\lesssim_M \lambda_{q+1}^{-M \epsilon_{q+1}/2} \norm{a_q}_{C_t^0 C_x^M} \lesssim  \lambda_{q+1}^{-M \epsilon_{q+1}/2}\\
     \norm{\P_{> \nu_{q+1}}(\rho_{q+1})}_{L_x^\infty}&\lesssim_M \lambda_{q+1}^{\frac{1-\epsilon_{q+1}}{k}-M \epsilon_{q+1}} 
\end{align*}
Now 
\begin{align*}
      \norm{\mathsf{R}_{q+1}}_{\dot{\mathbb{A}}_{x,t}^{-s}} &\lesssim \lambda_{q+1}^{1 - \epsilon_{q+1} - \frac{M\epsilon_{q+1}}{2}} + \lambda_{q+1}^{(1 - \epsilon_{q+1}) -M\epsilon_{q+1}} + \lambda_{q+1}^{(1- \epsilon_{q+1}) - \frac32 M \epsilon_{q+1}} 
\end{align*}
We may choose $M = M_{q+1}$ sufficiently large so that all the three exponents are negative. After this choice of $M$, we may choose $\lambda_{q+1}$ sufficiently large so that the total remainder is as small as required at stage $q+1$. Combining both estimates yields the desired estimate 
\begin{equation}\label{eq:EO-decay}
    \norm{E_O}_{\dot{\mathbb{A}}_{x,t}^{-s}}
    < 2^{-(q+1) - 200 - 2} .
\end{equation}

\subsubsection*{Temporal Error}

Recall 
\[E_T = \partial_t \omega_{q+1}\] where $\omega_{q+1}$ is the unique mean-free antiderivative of $w_{q+1}$, i.e.
\[\partial_x \omega_{q+1} = w_{q+1}, \qquad \hat{\omega}_{q+1}(t,0) = 0\]
And so
\begin{equation*}
    \hat{E}_T(t, \xi) = \frac{1}{2\pi i \xi} \partial_t \hat{w}_{q+1}(t, \xi).
\end{equation*}
Thus 
\begin{align}\label{eq:ET-startnew}
    \norm{E_T}_{\dot{\mathbb{A}}_{x,t}^{-s}} &= \sum_{\xi \neq 0}|\xi|^{-s} \norm{\hat{E}_T(\cdot, \xi)}_{C_t^0} \lesssim \sum_{\xi \neq 0} |\xi|^{-s - 1}\norm{\partial_t \hat{w}_{q+1}(\cdot, \xi)}_{C_t^0}
\end{align}
Now we have
\[
\|\partial_t\hat{w}_{q+1}(\cdot,\xi)\|_{C_t^0}
\le
\|\partial_t w_{q+1}\|_{C_t^0L_x^1},
\]
and so from \eqref{eq:ET-startnew} we obtain
\[
\|E_T\|_{\dot{\mathbb{A}}_{x,t}^{-s}}
\lesssim
\|\partial_t w_{q+1}\|_{C_t^0L_x^1}
\sum_{|\xi|\gtrsim \lambda_{q+1}^{\epsilon_{q+1}}} |\xi|^{-s-1}.
\]
For $s >0$ the series converges, thus
\begin{equation}
    \|E_T\|_{\dot{\mathbb{A}}_{x,t}^{-s}}
\lesssim  \|\partial_t w_{q+1}\|_{C_t^0L_x^1} \label{eq:ET-pre-dtwnew}
\end{equation}

It remains to estimate $\partial_t w_{q+1}$ in $C_t^0 L_x^1$. Differentiating in time gives 
\begin{align*}
    \partial_t w_{q+1}
&=
\mathbb P_{\le \mu_{q+1}}(\partial_t a_q)\,
\mathbb P_{\le \nu_{q+1}}(\rho_{q+1})\,
h_{q+1}
+
\mathbb P_{\le \mu_{q+1}^{1/2}}(a_q)\,
\mathbb P_{\le \nu_{q+1}}(\rho_{q+1})\,
h_{q+1}'.
\end{align*}

Using Lemma \ref{lem:proj}, H\"older's inequality, and the condition $0\le h_{q+1}\le 1$, we obtain
\[
\|\partial_t w_{q+1}\|_{C_t^0L_x^1}
\lesssim
\Bigl(
\|\partial_t a_q\|_{C_t^0L_x^\infty}
+
\|a_q\|_{C_t^0L_x^\infty}\| h_{q+1}'\|_{C_t^0}
\Bigr)
\|\rho_{q+1}\|_{L_x^1}.
\]
Now $\|\partial_t a_q\|_{C_t^0L_x^\infty}$ and $\|a_q\|_{C_t^0L_x^\infty}$ are $\lambda_{q+1}$-independent constants and so we may estimate them by $1$. By construction for the fixed parameter
\[
0<\beta<\frac{1}{16}\left(1-\frac{1}{k}\right)
\]
we have
\[
\|h_{q+1}'\|_{C_t^0}\lesssim \lambda^{\beta}_{q+1}.
\]

Finally, Lemma~\ref{lem:boldW} gives
\[
\|\rho_{q+1}\|_{L_x^1}
\lesssim
\lambda_{q+1}^{(1-\epsilon_{q+1})\left(\frac1k-1\right)}
=
\lambda_{q+1}^{-\frac{(1-\epsilon_{q+1})(k-1)}{k}}.
\]
Thus
\[
\|\partial_t w_{q+1}\|_{C_t^0L_x^1}
\lesssim_q
(1+\lambda^{\beta}_{q+1})
\lambda_{q+1}^{-\frac{(1-\epsilon_{q+1})(k-1)}{k}}
\lesssim_q
\lambda_{q+1}^{\beta-\frac{(1-\epsilon_{q+1})(k-1)}{k}}.
\]
Therefore substituting into \eqref{eq:ET-pre-dtwnew} yields
\begin{align*}
      \|E_T\|_{\dot{\mathbb{A}}_{x,t}^{-s}} &\lesssim \lambda_{q+1}^{\beta-\frac{(1-\epsilon_{q+1})(k-1)}{k}} 
\end{align*}
Since $\epsilon_{q+1} = 2^{-(q+1)}$ and $\beta < \frac{1}{16}\left(1-\frac{1}{k}\right)$, one may choose $\lambda_{q+1}$ large enough so that 
\begin{equation}\label{eq:ET-decay}
    \norm{E_T}_{\dot{\mathbb{A}}_{x,t}^{-s}}
    <
   2^{-(q+1) - 200 - 2}.
\end{equation}

Finally combining~\eqref{eq:ED-decay},~\eqref{eq:EN-decay},~\eqref{eq:EO-decay}, and~\eqref{eq:ET-decay} we deduce
$$
\Vert E_{q+1} \Vert_{\dot{\mathbb{A}}_{x,t}^{-s}} < 2^{-(q+1) - 200}
$$
which closes the induction.

\subsection{Proof of Item \ref{i:5}}\label{subsec:i5}

We suppose that item \ref{i:5} holds at level $q$ and aim to show this holds for level $q+1$ as well. Using that $u_{q+1} = w_{q+1} + u_q$ we have the expansion
\begin{equation}\label{eq:abs_sum_expansion}
    \begin{split}
        & \sum_{\xi_1 + \ldots + \xi_k \not = 0} \left\Vert  \hat{u}_{q+1}(t,\xi_1) \hat{u}_{q+1}(t,\xi_2) \cdots \hat{u}_{q+1}(t,\xi_k) \right\Vert_{C_t^0}|\xi_1 + \ldots + \xi_k|^{-s}\\
        &\leq \sum_{\xi_1 + \ldots + \xi_k \not = 0} \left\Vert  \hat{u}_{q}(t,\xi_1) \hat{u}_{q}(t,\xi_2) \cdots \hat{u}_{q}(t,\xi_k) \right\Vert_{C_t^0}|\xi_1 + \ldots + \xi_k|^{-s}\\
        &+ \sum_{m=1}^{k-1} \binom{k}{m} \sum_{\xi_1 + \ldots + \xi_k \not=0} \left\Vert \prod_{j=1}^m \hat{w}_{q+1}(t,\xi_j) \prod_{j=m+1}^k \hat{u}_q(t,\xi_j)\right\Vert_{C_t^0} |\xi_1 + \ldots + \xi_k|^{-s}\\
        &+ \sum_{\xi_1 + \ldots + \xi_k \not=0} \left\Vert \hat{w}_{q+1}(t,\xi_1) \hat{w}_{q+1}(t,\xi_2) \cdots \hat{w}_{q+1}(t,\xi_k)\right\Vert_{C_t^0} |\xi_1 + \ldots + \xi_k|^{-s}.
    \end{split}
\end{equation}

The term on the second line of ~\eqref{eq:abs_sum_expansion} is less than $C - 2^{-q-100}$ by hypothesis. To analyze the final two lines, we introduce the auxiliary quantities
\begin{equation*}
    AF(f,j) = \sum_{\xi_1 + \ldots+ \xi_j \not=0} \Vert \hat{f}(t,\xi_1) \cdots \hat{f}(t,\xi_j)\Vert_{C_t^0} |\xi_1 + \ldots + \xi_j|^{-s}
\end{equation*}
and
\begin{equation*}
    ZF(f,j) = \sum_{\xi_1 + \ldots+ \xi_j =0} \Vert \hat{f}(t,\xi_1) \cdots \hat{f}(t,\xi_j)\Vert_{C_t^0}.
\end{equation*}

Our first task will be to obtain bounds on $AF(w_{q+1},j)$ and $ZF(w_{q+1},j)$ for $1 \leq j \leq k$ since it is clear that
$$
AF(u_q,j), ZF(u_q,j) \lesssim 1.
$$
First from ~\eqref{eq:ak} we have the Taylor expansion
\begin{equation*}
    \mathbb{P}_{\leq \mu_{q+1}^{1/2}}\left( a_q(t,x) \right) = \sum_{n=0}^\infty \binom{1/k}{n} A_{q+1}^{1/k - n} \sigma^n \mathbb{P}_{\leq \mu_{q+1}^{1/2}} \left(E_q^n\right)
\end{equation*}
and from Lemma \ref{lem:fourier} we obtain
\begin{equation*}
    \mathbb{P}_{\leq \nu_{q+1}}\left(\rho_{q+1}\right) = \sum_{\xi \in \Z} \lambda_{q+1}^{(1-\epsilon_{q+1})\left(\frac{1}{k}-1\right)} \hat{\phi}_k\left(\lambda_{q+1}^{\epsilon_{q+1}-1}\xi\right) m_{\leq \mu_{q+1}}(\xi) e^{2\pi i \lambda_{q+1}^{\epsilon_{q+1}} \xi x}r
\end{equation*}
(recall $m_{\leq \mu_{q+1}}$ from Definition \ref{def:projs}) thus we have
\begin{equation*}
\begin{split}
    \hat{w}_{q+1}(t,\eta) &= \sum_{n=0}^\infty \sum_{\xi \in \Z} \lambda_{q+1}^{(1-\epsilon_{q+1})\left(\frac{1}{k}-1\right)} \hat{\phi}_k\left(\lambda_{q+1}^{\epsilon_{q+1}-1}\xi\right) \binom{1/k}{n} A_{q+1}^{1/k-n} \\
    &\times \sigma^n  \left(E_q^n h_{q+1}\right)^\wedge(t,\eta - \lambda_{q+1}^{\epsilon_{q+1}} \xi) m_{\leq \mu_{q+1}}(\xi)m_{\leq \mu_{q+1}^{1/2}}(\eta-\lambda_{q+1}^{\epsilon_{q+1}}\xi).
    \end{split}
\end{equation*}

With this, we may proceed with the following lemmas

\begin{lemma}\label{lem:horrible1}
    We have that
    \[
\begin{aligned}
AF(w_{q+1},j)
\lesssim_{q,j,s}\;&
\lambda_{q+1}^{(1-\epsilon_{q+1})(\frac jk-1)-s\epsilon_{q+1}}
A_{q+1}^{j/k}
\\
&+
\lambda_{q+1}^{(1-\epsilon_{q+1})(\frac jk-1)}
A_{q+1}^{\frac jk-1}
\|E_q\|_{\dot{\mathbb \A}^{-s}_{x, t}}
\\
&+
\lambda_{q+1}^{(1-\epsilon_{q+1})(\frac jk-1)-s\epsilon_{q+1}}
A_{q+1}^{\frac jk-1}
\\
&+
\lambda_{q+1}^{(1-\epsilon_{q+1})(\frac jk-1)}
A_{q+1}^{\frac jk-2}.
\end{aligned}
\]
In particular,
\[
AF(w_{q+1},j)
\lesssim_{q,j,s}
\lambda_{q+1}^{(1-\epsilon_{q+1})(\frac jk-1)}
A_{q+1}^{j/k}.
\]

\end{lemma}
\begin{proof}
    Clearly
    \begin{equation*}
        \begin{split}
            &\Vert \hat{w}_{q+1}(t,\eta_1) \cdots \hat{w}_{q+1}(t,\eta_j) \Vert_{C_t^0}\\
            &\lesssim \sum_{n_1,\ldots,n_j = 0}^\infty \sum_{\xi_1, \ldots,\xi_j \in \Z} \lambda_{q+1}^{j(1-\epsilon_{q+1})\left(\frac{1}{k}-1\right)} \left|\hat{\phi}_k\left(\lambda_{q+1}^{\epsilon_{q+1}-1}\xi_1\right) \cdots \hat{\phi}_k\left(\lambda_{q+1}^{\epsilon_{q+1}-1}\xi_j\right)\right| \left|\binom{1/k}{n_1} \cdots \binom{1/k}{n_j}\right| A_{q+1}^{\frac{j}{k}-n_1 -\ldots - n_j}\\
            &\times \left\Vert \left(E_q^{n_1} h_{q+1}\right)^{\wedge}(t,\eta_1 - \lambda_{q+1}^{\epsilon_{q+1}} \xi_1) \right\Vert_{C_t^0} \cdots \left\Vert \left(E_q^{n_1} h_{q+1}\right)^{\wedge}(t,\eta_1 - \lambda_{q+1}^{\epsilon_{q+1}} \xi_j) \right\Vert_{C_t^0}
        \end{split}
    \end{equation*}
    There are four cases we consider.

    \textbf{Case 1:} $n_1 = n_2 = \ldots =n_j = 0$.\\
    In this case we have that
    $$
    \left(E_q^0 h_{q+1}\right)^{\wedge}(t,\eta_i - \lambda_{q+1}^{\epsilon_{q+1}} \xi_i) = \delta(\eta_i - \lambda_{q+1}^{\epsilon_{q+1}} \xi_i) h_{q+1}(t)
    $$
    and thus
    $$
    \left\Vert \left(E_q^{0} h_{q+1}\right)^{\wedge}(t,\eta_i - \lambda_{q+1}^{\epsilon_{q+1}} \xi_i) \right\Vert_{C_t^0} = \left\Vert \delta(\eta_i - \lambda_{q+1}^{\epsilon_{q+1}} \xi_i) h_{q+1}(t) \right\Vert_{C_t^0} = \delta(\eta_i - \lambda_{q+1}^{\epsilon_{q+1}} \xi_i)
    $$
    for all $1 \leq i \leq j$. Hence
    \begin{equation}\label{eq:n=0_case}
        \begin{split}
            &\sum_{\eta_1 + \ldots + \eta_j \not=0} \sum_{\xi_1,\ldots,\xi_j \in \Z} \lambda_{q+1}^{j(1-\epsilon_{q+1})\left(\frac{1}{k} - 1\right)} \left|\hat{\phi}_k\left(\lambda_{q+1}^{\epsilon_{q+1}-1}\xi_1\right) \cdots \hat{\phi}_k\left(\lambda_{q+1}^{\epsilon_{q+1}-1}\xi_j\right)\right| A_{q+1}^{j/k}\\
            &\times \delta(\eta_1 - \lambda_{q+1}^{\epsilon_{q+1}} \xi_1) \cdots \delta(\eta_j - \lambda_{q+1}^{\epsilon_{q+1}} \xi_j) |\eta_1 + \ldots + \eta_j|^{-s}\\
            &= \sum_{\xi_1 + \ldots + \xi_j \not=0} \lambda_{q+1}^{j(1-\epsilon_{q+1})\left(\frac{1}{k}-1\right) - s\epsilon_{q+1}} \left| \hat{\phi}_k\left(\lambda_{q+1}^{\epsilon_{q+1}-1}\xi_1\right) \cdots \hat{\phi}_k\left(\lambda_{q+1}^{\epsilon_{q+1}-1}\xi_j\right)\right| A_{q+1}^{j/k} |\xi_1 + \ldots + \xi_j|^{-s}\\
            &= \lambda_{q+1}^{(1-\epsilon_{q+1})\left(\frac{j}{k}-1\right)-s\epsilon_{q+1}} A_{q+1}^{j/k} \sum_{\Xi \not=0} |\Xi|^{-s} \sum_{\xi_1,\ldots,\xi_{j-1}} \left|\prod_{i=1}^{j-1}\hat{\phi}_k\left(\lambda_{q+1}^{\epsilon_{q+1}-1}\xi_i\right) \hat{\phi}_k\left(\lambda_{q+1}^{\epsilon_{q+1}-1}\left(\Xi - \xi_1 - \ldots - \xi_{j-1}\right)\right)\right|\\
            &\times \lambda_{q+1}^{(j-1)(\epsilon_{q+1}-1)}\\
            &\lesssim \lambda_{q+1}^{(1-\epsilon_{q+1})\left(\frac{j}{k}-1\right)-s\epsilon_{q+1}} A_{q+1}^{j/k} \sum_{\Xi \not=0} |\Xi|^{-s} \int_{\R^{j-1}} \left| \prod_{i=1}^{j-1} \hat{\phi}_k(\xi_i) \hat{\phi}_k(\lambda_{q+1}^{\ep -1}\Xi - \xi_1 - \ldots - \xi_{j-1})\right|\, d\xi_1\, \ldots\, d\xi_{j-1}\\
            &\lesssim \lambda_{q+1}^{(1-
            \epsilon_{q+1})\left(\frac{j}{k}-1\right)-s\epsilon_{q+1}} A_{q+1}^{j/k}
        \end{split}
    \end{equation}
    where we have utilized that $s > 1$, our choice of taking $\lambda_{q+1}$ arbitrarily large, and the generalized Young's convolution inequality to get that
    $$
    \left\Vert \left(|\hat{\phi}_k| \ast |\hat{\phi}_k| \ast \ldots \ast |\hat{\phi}_k|\right)\left(\lambda_{q+1}^{\epsilon_{q+1}-1} \cdot\right)\right\Vert_{L^\infty} \lesssim \left\Vert \hat{\phi}_k\right\Vert_{L^1}^{j-1} \left\Vert \hat{\phi}_k \right\Vert_{L^\infty} \lesssim 1.
    $$

    \textbf{Case 2:} $n_i = 1$ for some $1 \leq i \leq j$, $n_m=0$ for all $m \not=i$, $\xi_1 + \ldots + \xi_j = 0$.

    By symmetry considerations let us suppose that $i = 1$; i.e. $n_1 = 1$ and $n_2=\ldots =n_j =0$. Then we have that
    $$
    \left\Vert \left(E_qh_{q+1}\right)^{\wedge}(t,\eta_1 - \lambda_{q+1}^{\epsilon_{q+1}} \xi_1)\right\Vert_{C_t^0} = \left\Vert \hat{E}_q(t,\eta_1 - \lambda_{q+1}^{\epsilon_{q+1}} \xi_1)\right\Vert_{C_t^0}
    $$
   as well as
   $$
   \left\Vert \left(E_q^{0} h_{q+1}\right)^{\wedge}(t,\eta_i - \lambda_{q+1}^{\epsilon_{q+1}} \xi_i) \right\Vert_{C_t^0} = \delta(\eta_i - \lambda_{q+1}^{\epsilon_{q+1}} \xi_i)
   $$
   for all $1 < i \leq j$. Hence, using the same convolution estimate from Case 1, we have
   \begin{equation}\label{eq:n=1_case}
       \begin{split}
           &\sum_{\eta_1 + \ldots + \eta_j \not=0} \sum_{\xi_1 + \ldots + \xi_j=0} \lambda_{q+1}^{j(1-\epsilon_{q+1})\left(\frac{1}{k} - 1\right)} \left|\hat{\phi}_k\left(\lambda_{q+1}^{\epsilon_{q+1}-1}\xi_1\right) \cdots \hat{\phi}_k\left(\lambda_{q+1}^{\epsilon_{q+1}-1}\xi_j\right)\right| A_{q+1}^{\frac{j}{k}-1}\\
           &\times \left\Vert \hat{E}_q(t,\eta_1 - \lambda_{q+1}^{\epsilon_{q+1}} \xi_1)\right\Vert_{C_t^0}\delta(\eta_2 - \lambda_{q+1}^{\epsilon_{q+1}} \xi_2) \cdots \delta(\eta_j - \lambda_{q+1}^{\epsilon_{q+1}} \xi_j) |\eta_1 + \ldots + \eta_j|^{-s}\\
           &= \sum_{\xi_1 + \ldots + \xi_j = 0}  \sum_{\eta_1 \not= \lambda_{q+1}^{\epsilon_{q+1}} \xi_1} \lambda_{q+1}^{j(1-\epsilon_{q+1})\left(\frac{1}{k}-1\right)}A_{q+1}^{\frac{j}{k}-1} \left|\hat{\phi}_k\left(\lambda_{q+1}^{\epsilon_{q+1}-1}\xi_1\right) \cdots \hat{\phi}_k\left(\lambda_{q+1}^{\epsilon_{q+1}-1}\xi_j\right)\right|\\
           &\times \left\Vert \hat{E}_q(t,\eta_1 - \lambda_{q+1}^{\epsilon_{q+1}} \xi_1)\right\Vert_{C_t^0} \left|\eta_1 -\lambda_{q+1}^{\epsilon_{q+1}} \xi_1\right|^{-s}\\
           &= \sum_{\xi_1 + \ldots + \xi_j = 0} \lambda_{q+1}^{j(1-\epsilon_{q+1})\left(\frac{1}{k}-1\right)} A_{q+1}^{\frac{j}{k}-1} \left|\hat{\phi}_k\left(\lambda_{q+1}^{\epsilon_{q+1}-1}\xi_1\right) \cdots \hat{\phi}_k\left(\lambda_{q+1}^{\epsilon_{q+1}-1}\xi_j\right)\right| \Vert E_q \Vert_{\dot{\mathbb{A}}_{x,t}^{-s}}\\
           &\lesssim \lambda_{q+1}^{(1-\epsilon_{q+1})\left(\frac{j}{k}-1\right)} A_{q+1}^{\frac{j}{k}-1} \Vert E_q \Vert_{\dot{\mathbb{A}}_{x,t}^{-s}}.
       \end{split}
   \end{equation}
    
    \textbf{Case 3:} $n_i = 1$ for some $1 \leq i \leq j$, $n_m=0$ for all $m \not=i$, $\xi_1 + \ldots + \xi_j \not= 0$.

    Again from symmetry considerations we assume that $n_1 = 1$ and $n_2 = \ldots = n_j = 0$. From item \ref{i:3} and the fact $(u_q,E_q)$ solve \eqref{eq:relaxed}, the frequency support of $E_q$ is contained in $B(0,k\lambda_q^{3/2})$. Hence we have that
    $$
    |\eta_1 - \lambda_{q+1}^{\epsilon_{q+1}} \xi_1| \lesssim \lambda_q^{3/2}.
    $$
    Since $\eta_i = \lambda_{q+1}^{\epsilon_{q+1}} \xi_i$ for all $2 \leq i \leq j$ then we see that we may choose $\lambda_{q+1}$ large enough to ensure that
    \begin{equation}\label{eq:AF-case3-gap}
        \begin{split}
            |\eta_1 + \eta_2 + \ldots + \eta_j| &= \left|\eta_1 - \lambda_{q+1}^{\epsilon_{q+1}} \xi_1 + \lambda_{q+1}^{\epsilon_{q+1}}\left(\xi_1 + \xi_2 + \ldots + \xi_j\right)\right|\\
            &\geq \lambda_{q+1}^{\epsilon_{q+1}} \left|\xi_1 + \ldots + \xi_j\right| - |\eta_1 - \lambda_{q+1}^{\epsilon_{q+1}} \xi_1|\\
            &\geq \frac{1}{2}\lambda_{q+1}^{\epsilon_{q+1}} \left|\xi_1 + \ldots + \xi_j\right|\\
            &\simeq \lambda_{q+1}^{\epsilon_{q+1}} \left|\xi_1 + \ldots + \xi_j\right|
        \end{split}
    \end{equation}
    Hence again applying the same convolution estimate we have
    \begin{equation}\label{eq:n=1_case_2}
        \begin{split}
            &\sum_{\eta_1 + \ldots + \eta_j \not=0} \sum_{\xi_1 + \ldots + \xi_j \not= 0} \lambda_{q+1}^{j(1-\epsilon_{q+1})\left(\frac{1}{k} - 1\right)} \left|\hat{\phi}_k\left(\lambda_{q+1}^{\epsilon_{q+1}-1}\xi_1\right) \cdots \hat{\phi}_k\left(\lambda_{q+1}^{\epsilon_{q+1}-1}\xi_j\right)\right| A_{q+1}^{\frac{j}{k}-1}\\
           &\times \left\Vert \hat{E}_q(t,\eta_1 - \lambda_{q+1}^{\epsilon_{q+1}} \xi_1)\right\Vert_{C_t^0}\delta(\eta_2 - \lambda_{q+1}^{\epsilon_{q+1}} \xi_2) \cdots \delta(\eta_j - \lambda_{q+1}^{\epsilon_{q+1}} \xi_j) |\eta_1 + \ldots + \eta_j|^{-s}\\
           &\lesssim  \sum_{\xi_1 + \ldots + \xi_j \not=0} \sum_{\eta_1 \not= \lambda_{q+1}^{\epsilon_{q+1}} \xi_1} \lambda_{q+1}^{j(1-\epsilon_{q+1})\left(\frac{1}{k}-1\right)-s\epsilon_{q+1}} A_{q+1}^{\frac{j}{k}-1} \left|\hat{\phi}_k\left(\lambda_{q+1}^{\epsilon_{q+1}-1}\xi_1\right) \cdots \hat{\phi}_k\left(\lambda_{q+1}^{\epsilon_{q+1}-1}\xi_j\right)\right|\\
           &\times \left\Vert \hat{E}_q(t,\eta_1 - \lambda_{q+1}^{\epsilon_{q+1}} \xi_1) \right\Vert_{C_t^0} |\xi_1 + \ldots + \xi_j|^{-s}\\
           &\lesssim \lambda_{q+1}^{(1-\epsilon_{q+1})\left(\frac{j}{k}-1\right) - s\epsilon_{q+1}} A_{q+1}^{\frac{j}{k}-1} \left(\sum_{\mu \not=0} \left\Vert \hat{E}_q(t,\mu) \right\Vert_{C_t^0}\right)\\
           &\lesssim \lambda_{q+1}^{(1-\epsilon_{q+1})\left(\frac{j}{k}-1\right) - s\epsilon_{q+1}} A_{q+1}^{\frac{j}{k}-1}
        \end{split}
    \end{equation}
    \textbf{Case 4:} $n_1 + n_2 + \ldots + n_j \geq 2$.
    Since $h_{q+1}\equiv 1$ on $\operatorname{supp}_t E_q$, we have
\[
    E_q^{n_i}h_{q+1}=E_q^{n_i}
    \qquad \text{whenever } n_i\geq 1,
\]
while for $n_i=0$,
\[
    \left(E_q^0h_{q+1}\right)^\wedge(t,\mu)
    =
    \delta_{\mu,0}h_{q+1}(t).
\]
Define
$$
C_q = \sum_{\xi \in \Z} \Vert \hat{E}_q(t,\xi) \Vert_{C_t^0} = \Vert E_q \Vert_{\A^0_{x,t}}.
$$
This quantity is finite due to the compact frequency support of $E_q$ discussed earlier and is clearly independent of $q+1$. Now using the algebra property of $\A^{0}_{x,t}$, we have
\[
    \|E_q^{n_i}h_{q+1}\|_{\dot{\mathbb{A}}_{x,t}^{-s}} < \Vert E_q^{n_i} \Vert_{\A^0_{x,t}}
    \lesssim C_q^{n_i}
\]
for every $n_i\geq0$.

Fix $n_1,\ldots,n_j\geq0$ with $n_1+\cdots+n_j\geq2$. We split according to whether $\xi_1+\cdots+\xi_j=0$, or not.

If $\xi_1+\cdots+\xi_j=0$, then
\[
    \eta_1+\cdots+\eta_j
    =
    \sum_{i=1}^j
    \left(\eta_i-\lambda_{q+1}^{\epsilon_{q+1}}\xi_i\right).
\]
After the change of variables $\mu_i=\eta_i-\lambda_{q+1}^{\epsilon_{q+1}}\xi_i$, we get
\[
\begin{aligned}
&\sum_{\eta_1+\cdots+\eta_j\neq0}
\prod_{i=1}^j
\left\|
\left(E_q^{n_i}h_{q+1}\right)^\wedge
\left(t,\eta_i-\lambda_{q+1}^{\epsilon_{q+1}}\xi_i\right)
\right\|_{C_t^0}
|\eta_1+\cdots+\eta_j|^{-s}
\\
&\qquad\le
\prod_{i=1}^j\sum_{\mu_i \in \Z}
\|(E_q^{n_i}h_{q+1})^\wedge(t,\mu_i)\|_{C_t^0}
\lesssim_q
C_q^{n_1+\cdots+n_j}.
\end{aligned}
\]
Therefore, using the same convolution estimate as in Case 2,
\begin{equation}\label{eq:page_29_est}
    \begin{split}
        &\sum_{\xi_1+\cdots+\xi_j=0}
\lambda_{q+1}^{j(1-\epsilon_{q+1})(\frac1k-1)}
A_{q+1}^{\frac jk-(n_1+\cdots+n_j)}
\left|\binom{1/k}{n_1}\cdots\binom{1/k}{n_j}\right|
\\
&\quad\times
\prod_{i=1}^j
\left|
\hat{\phi}_k
\left(\lambda_{q+1}^{\epsilon_{q+1}-1}\xi_i\right)
\right|
C_q^{n_1+\cdots+n_j}
\\
&\lesssim_{q,j,s}
\left|\binom{1/k}{n_1}\cdots\binom{1/k}{n_j}\right|
\lambda_{q+1}^{j(1-\epsilon_{q+1})(\frac1k-1)}
\lambda_{q+1}^{(j-1)(1-\epsilon_{q+1})}
A_{q+1}^{\frac jk-(n_1+\cdots+n_j)}
C_q^{n_1+\cdots+n_j}
\\
&=
\left|\binom{1/k}{n_1}\cdots\binom{1/k}{n_j}\right|
\lambda_{q+1}^{(1-\epsilon_{q+1})(\frac jk-1)}
A_{q+1}^{\frac jk}
\left(\frac{C_q}{A_{q+1}}\right)^{n_1+\cdots+n_j}.
    \end{split}
\end{equation}

Now suppose $\xi_1+\cdots+\xi_j\neq0$. Then, on the support of the frequency cutoff factors appearing in $\hat{w}_{q+1}$, each summand vanishes unless $\left|\eta_i-\lambda_{q+1}^{\epsilon_{q+1}}\xi_i \right| \lesssim \mu_{q+1}^{1/2}$ for $i=1,\cdots,j$. Hence,

\begin{equation}\label{eq:page29_est}
|\eta_1 +\cdots +\eta_j| \ge \lambda_{q+1}^{\epsilon_{q+1}}|\xi_1+ \cdots +\xi_j| - \sum_{i=1}^j \left(\eta_i-\lambda_{q+1}^{\epsilon_{q+1}}\xi_i\right) \ge \frac12 \lambda_{q+1}^{\epsilon_{q+1}}|\xi+\cdots+\xi_j|
\end{equation}
provided $\lambda_{q+1}$ is sufficiently large. Therefore
\[
|\eta_1 +\cdots +\eta_j|^{-s} \lesssim_s \lambda_{q+1}^{-s\epsilon_{q+1}}|\xi_1 + \cdots +\xi_j|^{-s}.
\]
Consequently,
\[
\begin{aligned}
    &\sum_{\eta_1,\cdots,\eta_j}\prod_{i=1}^j \Vert(E_q^{n_i}h_{q+1})^\wedge(t,\eta_i - \lambda_{q+1}^{\epsilon_{q+1}}\xi_i)\Vert_{C_t^0} |\eta_1 + \cdots + \eta_j|^{-s}\\
    &\qquad \lesssim_{j,s} \lambda_{q+1}^{-s\epsilon_{q+1}}|\xi_1 + \cdots + \xi_j|^{-s} \prod_{i=1}^j \sum_{\eta_i \in \Z} \Vert (E^{n_i}_q h_{q+1})^\wedge (t,\eta_i - \lambda_{q+1}^{\epsilon_{q+1}}\xi_i)\Vert_{C_t^0}\\
    &\qquad \lesssim_{q,j,s} \lambda_{q+1}^{-s\epsilon_{q+1}}|\xi_1 + \cdots + \xi_j|^{-s} C^{n_1+\cdots+n_j}_q.
\end{aligned}
\]
Using the same convolution estimate as in Case 1, we obtain
\begin{equation}\label{eq:page_29_est_3}
    \begin{split}
        &\sum_{\xi_1+\cdots+\xi_j\neq0}
\lambda_{q+1}^{j(1-\epsilon_{q+1})(\frac1k-1)}
A_{q+1}^{\frac jk-(n_1+\cdots+n_j)}
\left|\binom{1/k}{n_1}\cdots\binom{1/k}{n_j}\right|
\\
&\quad\times
\prod_{i=1}^j
\left|
\hat{\phi}_k
\left(\lambda_{q+1}^{\epsilon_{q+1}-1}\xi_i\right)
\right|
\lambda_{q+1}^{-s\epsilon_{q+1}}
|\xi_1+\cdots+\xi_j|^{-s}
C_q^{n_1+\cdots+n_j}
\\
&\lesssim_{q,j,s}
\left|\binom{1/k}{n_1}\cdots\binom{1/k}{n_j}\right|
\lambda_{q+1}^{j(1-\epsilon_{q+1})(\frac1k-1)-s\epsilon_{q+1}}
\lambda_{q+1}^{(j-1)(1-\epsilon_{q+1})}
A_{q+1}^{\frac jk-(n_1+\cdots+n_j)}
C_q^{n_1+\cdots+n_j}
\\
&=
\left|\binom{1/k}{n_1}\cdots\binom{1/k}{n_j}\right|
\lambda_{q+1}^{(1-\epsilon_{q+1})(\frac jk-1)-s\epsilon_{q+1}}
A_{q+1}^{\frac jk}
\left(\frac{C_q}{A_{q+1}}\right)^{n_1+\cdots+n_j}
\\
&\leq
\left|\binom{1/k}{n_1}\cdots\binom{1/k}{n_j}\right|
\lambda_{q+1}^{(1-\epsilon_{q+1})(\frac jk-1)}
A_{q+1}^{\frac jk}
\left(\frac{C_q}{A_{q+1}}\right)^{n_1+\cdots+n_j}.
    \end{split}
\end{equation}

Combining the two subcases, we obtain
\[
\begin{aligned}
&\sum_{\eta_1+\cdots+\eta_j\neq0}
\sum_{\substack{n_1,\ldots,n_j\geq0\\ n_1+\cdots+n_j\geq2}}
\sum_{\xi_1,\ldots,\xi_j\in\mathbb Z}
\lambda_{q+1}^{j(1-\epsilon_{q+1})(\frac1k-1)}
A_{q+1}^{\frac jk-(n_1+\cdots+n_j)}
\left|\binom{1/k}{n_1}\cdots\binom{1/k}{n_j}\right|
\\
&\quad\times
\prod_{i=1}^j
\left|
\hat{\phi}_k
\left(\lambda_{q+1}^{\epsilon_{q+1}-1}\xi_i\right)
\right|
\prod_{i=1}^j
\left\|
\left(E_q^{n_i}h_{q+1}\right)^\wedge
\left(t,\eta_i-\lambda_{q+1}^{\epsilon_{q+1}}\xi_i\right)
\right\|_{C_t^0}
|\eta_1+\cdots+\eta_j|^{-s}
\\
&\lesssim_{q,j,s}
\lambda_{q+1}^{(1-\epsilon_{q+1})(\frac jk-1)}
A_{q+1}^{\frac jk}
\sum_{\substack{n_1,\ldots,n_j\geq0\\ n_1+\cdots+n_j\geq2}}
\left|\binom{1/k}{n_1}\cdots\binom{1/k}{n_j}\right|
\left(\frac{C_q}{A_{q+1}}\right)^{n_1+\cdots+n_j}.
\end{aligned}
\]
Choosing $A_{q+1}$ sufficiently large so that
\begin{equation}\label{eq:A_q+1_cond}
    \frac{C_q}{A_{q+1}}\leq \frac{1}{2},
\end{equation}
the absolute convergence of the binomial series gives
\[
\sum_{\substack{n_1,\ldots,n_j\geq0\\ n_1+\cdots+n_j\geq2}}
\left|\binom{1/k}{n_1}\cdots\binom{1/k}{n_j}\right|
\left(\frac{C_q}{A_{q+1}}\right)^{n_1+\cdots+n_j}
\lesssim_{j,k}
\left(\frac{C_q}{A_{q+1}}\right)^2.
\]
Therefore
\[
\begin{aligned}
&\sum_{\eta_1+\cdots+\eta_j\neq0}
\sum_{\substack{n_1,\ldots,n_j\geq0\\ n_1+\cdots+n_j\geq2}}
\sum_{\xi_1,\ldots,\xi_j\in\mathbb Z}
\lambda_{q+1}^{j(1-\epsilon_{q+1})(\frac1k-1)}
A_{q+1}^{\frac jk-(n_1+\cdots+n_j)}
\left|\binom{1/k}{n_1}\cdots\binom{1/k}{n_j}\right|
\\
&\quad\times
\prod_{i=1}^j
\left|
\hat{\phi}_k
\left(\lambda_{q+1}^{\epsilon_{q+1}-1}\xi_i\right)
\right|
\prod_{i=1}^j
\left\|
\left(E_q^{n_i}h_{q+1}\right)^\wedge
\left(t,\eta_i-\lambda_{q+1}^{\epsilon_{q+1}}\xi_i\right)
\right\|_{C_t^0}
|\eta_1+\cdots+\eta_j|^{-s}
\\
&\lesssim_{q,j,s}
\lambda_{q+1}^{(1-\epsilon_{q+1})(\frac jk-1)}
A_{q+1}^{\frac jk}
\left(\frac{C_q}{A_{q+1}}\right)^2
\\
&\lesssim_{q,j,s}
\lambda_{q+1}^{(1-\epsilon_{q+1})(\frac jk-1)}
A_{q+1}^{\frac jk-2}.
\end{aligned}
\] 
Combining Cases 1-4, we obtain
\[
\begin{aligned}
AF(w_{q+1},j)
\lesssim_{q,j,s}\;&
\lambda_{q+1}^{(1-\epsilon_{q+1})(\frac jk-1)-s\epsilon_{q+1}}
A_{q+1}^{j/k}
\\
&+
\lambda_{q+1}^{(1-\epsilon_{q+1})(\frac jk-1)}
A_{q+1}^{\frac jk-1}
\|E_q\|_{\dot{\mathbb A}^{-s}_xC_t^0}
\\
&+
\lambda_{q+1}^{(1-\epsilon_{q+1})(\frac jk-1)-s\epsilon_{q+1}}
A_{q+1}^{\frac jk-1}
\\
&+
\lambda_{q+1}^{(1-\epsilon_{q+1})(\frac jk-1)}
A_{q+1}^{\frac jk-2}.
\end{aligned}
\]
Since $A_{q+1}\geq1$, $\lambda_{q+1}^{-s\epsilon_{q+1}}\leq1$, and
$\|E_q\|_{\mathbb{A}_{x,t}^{0}}$ is a $\lambda_{q+1}$-independent constant, we conclude
\[
AF(w_{q+1},j)
\lesssim_{q,j,s}
\lambda_{q+1}^{(1-\epsilon_{q+1})(\frac jk-1)}
A_{q+1}^{j/k}.
\]
\end{proof}
\begin{lemma}\label{lem:horrible2}
For every $1\leq j\leq k$, we have
\[
    ZF(w_{q+1},j)
    \lesssim_{q,j}
    \lambda_{q+1}^{(1-\epsilon_{q+1})(\frac jk-1)}
    A_{q+1}^{j/k}.
\]
\end{lemma}

\begin{proof}
Recall that
\[
ZF(w_{q+1},j)
=
\sum_{\eta_1+\cdots+\eta_j=0}
\left\|
\hat{w}_{q+1}(t,\eta_1)\cdots
\hat{w}_{q+1}(t,\eta_j)
\right\|_{C_t^0}.
\]
Using the same expansion as in the previous lemma, we have
\[
\begin{aligned}
ZF(w_{q+1},j)
&\lesssim
\sum_{\eta_1+\cdots+\eta_j=0}
\sum_{n_1,\ldots,n_j=0}^{\infty}
\sum_{\xi_1,\ldots,\xi_j\in\mathbb Z}
\lambda_{q+1}^{j(1-\epsilon_{q+1})(\frac1k-1)}
A_{q+1}^{\frac jk-(n_1+\cdots+n_j)}
\\
&\quad\times
\left|
\binom{1/k}{n_1}\cdots\binom{1/k}{n_j}
\right|
\prod_{i=1}^j
\left|
\hat{\phi}_k
\left(\lambda_{q+1}^{\epsilon_{q+1}-1}\xi_i\right)
\right|
\\
&\quad\times
\prod_{i=1}^j
\left\|
\left(E_q^{n_i}h_{q+1}\right)^\wedge
\left(t,\eta_i-\lambda_{q+1}^{\epsilon_{q+1}}\xi_i\right)
\right\|_{C_t^0}.
\end{aligned}
\]

We first record the stage-$q$ bound used below. Since $h_{q+1}\equiv1$ on
$\operatorname{supp}_t E_q$, we have
\[
    E_q^{n_i}h_{q+1}=E_q^{n_i}
    \qquad \text{whenever } n_i\geq1,
\]
while for $n_i=0$,
\[
    \left(E_q^0h_{q+1}\right)^\wedge(t,\mu)
    =
    \delta(\mu)h_{q+1}(t).
\]
Recall that we had chosen
$$
C_q = \Vert E_q \Vert_{\A_{x,t}^0}.
$$
and thus
\[
    \|E_q^{n_i}h_{q+1}\|_{{\mathbb{A}}_{x,t}^{0}}
    \lesssim
    C_q^{n_i}
\]
for every $n_i\geq0$, with an implicit constant independent of $n_i$.

Fix $n_1,\ldots,n_j\geq0$. For fixed $\xi_1,\ldots,\xi_j$, make the change of variables
\[
    \mu_i
    =
    \eta_i-\lambda_{q+1}^{\epsilon_{q+1}}\xi_i.
\]
The constraint $\eta_1+\cdots+\eta_j=0$ becomes
\[
    \lambda_{q+1}^{\epsilon_{q+1}}(\xi_1+\cdots+\xi_j)
    +
    \mu_1+\cdots+\mu_j
    =
    0.
\]
Therefore the contribution of this fixed choice of $n_1,\ldots,n_j$ is bounded by
\[
\begin{aligned}
&\lambda_{q+1}^{j(1-\epsilon_{q+1})(\frac1k-1)}
A_{q+1}^{\frac jk-(n_1+\cdots+n_j)}
\left|
\binom{1/k}{n_1}\cdots\binom{1/k}{n_j}
\right|
\\
&\quad\times
\sum_{\mu_1,\ldots,\mu_j\in\mathbb Z}
\prod_{i=1}^j
\left\|
\left(E_q^{n_i}h_{q+1}\right)^\wedge(t,\mu_i)
\right\|_{C_t^0}
\\
&\quad\times
\sum_{\substack{\xi_1,\ldots,\xi_j\in\mathbb Z\\
\lambda_{q+1}^{\epsilon_{q+1}}(\xi_1+\cdots+\xi_j)
+
\mu_1+\cdots+\mu_j=0}}
\prod_{i=1}^j
\left|
\hat{\phi}_k
\left(\lambda_{q+1}^{\epsilon_{q+1}-1}\xi_i\right)
\right|.
\end{aligned}
\]

We claim that, uniformly in $\Xi\in\mathbb Z$,
\[
\sum_{\xi_1+\cdots+\xi_j=\Xi}
\prod_{i=1}^j
\left|
\hat{\phi}_k
\left(\lambda_{q+1}^{\epsilon_{q+1}-1}\xi_i\right)
\right|
\lesssim_j
\lambda_{q+1}^{(j-1)(1-\epsilon_{q+1})}.
\]
Indeed, writing
\[
    \xi_j=\Xi-\xi_1-\cdots-\xi_{j-1},
\]
we get for $\lambda_{q+1}$ chosen large enough
\begin{equation}\label{eq:page_32_conv_est}
    \begin{split}
        &\lambda_{q+1}^{(j-1)(\epsilon_{q+1}-1)}
\sum_{\xi_1,\ldots,\xi_{j-1}}
\prod_{i=1}^{j-1}
\left|
\hat{\phi}_k
\left(\lambda_{q+1}^{\epsilon_{q+1}-1}\xi_i\right)
\right|
\\
&\qquad\qquad\times
\left|
\hat{\phi}_k
\left(
\lambda_{q+1}^{\epsilon_{q+1}-1}\Xi
-
\lambda_{q+1}^{\epsilon_{q+1}-1}\xi_1
-\cdots
-\lambda_{q+1}^{\epsilon_{q+1}-1}\xi_{j-1}
\right)
\right|
\\
&\lesssim_j
\sup_{Y\in\mathbb R}
\int_{\mathbb R^{j-1}}
\prod_{i=1}^{j-1}
|\hat{\phi}_k(y_i)|
\left|
\hat{\phi}_k
\left(
Y-y_1-\cdots-y_{j-1}
\right)
\right|
\,dy_1\cdots dy_{j-1}
\\
&\lesssim_j
\|\hat{\phi}_k\|_{L^1}^{j-1}
\|\hat{\phi}_k\|_{L^\infty}
\lesssim_j 1.
    \end{split}
\end{equation}
Multiplying by $\lambda_{q+1}^{(j-1)(1-\epsilon_{q+1})}$ gives the claim.

Now fix $\mu_1$,$\ldots$, $\mu_j$. The constraint
\[
\lambda_{q+1}^{\epsilon_{q+1}}(\xi_1+\cdots+\xi_j)
+
\mu_1+\cdots+\mu_j=0
\]
either has no solutions, or else imposes
\[
    \xi_1+\cdots+\xi_j
    =
    -\lambda_{q+1}^{-\epsilon_{q+1}}(\mu_1+\cdots+\mu_j),
\]
which is a fixed integer value. Hence the preceding estimate gives
\begin{equation*}
\sum_{\substack{\xi_1,\ldots,\xi_j\in\mathbb Z\\
\lambda_{q+1}^{\epsilon_{q+1}}(\xi_1+\cdots+\xi_j)
+
\mu_1+\cdots+\mu_j=0}}
\prod_{i=1}^j
\left|
\hat{\phi}_k
\left(\lambda_{q+1}^{\epsilon_{q+1}-1}\xi_i\right)
\right| \lesssim_j
\lambda_{q+1}^{(j-1)(1-\epsilon_{q+1})}.
\end{equation*}

Therefore the contribution of this fixed choice of $n_1,\ldots,n_j$ is bounded by
\[
\begin{aligned}
&\lambda_{q+1}^{j(1-\epsilon_{q+1})(\frac1k-1)}
\lambda_{q+1}^{(j-1)(1-\epsilon_{q+1})}
A_{q+1}^{\frac jk-(n_1+\cdots+n_j)}
\left|
\binom{1/k}{n_1}\cdots\binom{1/k}{n_j}
\right|
\\
&\quad\times
\prod_{i=1}^j
\|E_q^{n_i}h_{q+1}\|_{\mathbb A_x^0C_t^0}
\\
&\lesssim_{q,j}
\left|
\binom{1/k}{n_1}\cdots\binom{1/k}{n_j}
\right|
\lambda_{q+1}^{(1-\epsilon_{q+1})(\frac jk-1)}
A_{q+1}^{j/k}
\left(\frac{C_q}{A_{q+1}}\right)^{n_1+\cdots+n_j}.
\end{aligned}
\]

Summing over $n_1,\ldots,n_j\geq0$, we obtain
\[
\begin{aligned}
ZF(w_{q+1},j)
&\lesssim_{q,j}
\lambda_{q+1}^{(1-\epsilon_{q+1})(\frac jk-1)}
A_{q+1}^{j/k}
\\
&\quad\times
\sum_{n_1,\ldots,n_j=0}^{\infty}
\left|
\binom{1/k}{n_1}\cdots\binom{1/k}{n_j}
\right|
\left(\frac{C_q}{A_{q+1}}\right)^{n_1+\cdots+n_j}.
\end{aligned}
\]
Recall from \eqref{eq:A_q+1_cond} we had chosen $A_{q+1}$ large enough so that
\[
    \frac{C_q}{A_{q+1}}\leq \frac12,
\]
and hence the binomial series converges absolutely. Therefore
\[
ZF(w_{q+1},j)
\lesssim_{q,j}
\lambda_{q+1}^{(1-\epsilon_{q+1})(\frac jk-1)}
A_{q+1}^{j/k}.
\]
\end{proof}
We also record the exact coefficient of the $E_q$-term which appears in the pure $w_{q+1}$ contribution when $j=k$. Inspecting Case 2 in the proof of the $AF(w_{q+1},j)$ estimate, this coefficient is 
\begin{equation}\label{eq:E-coeff-pure-w}
K_{q+1} := \sum_{\xi_1+\cdots+\xi_k=0} \prod_{i=1}^k \left| (\mathbb P_{\le \nu_{q+1}}\rho_{q+1})^{\wedge}(\xi_i)\right|    
\end{equation}
Indeed, in the case $j=k$, the binomial factor is exactly
\[
k\left|\binom{1/k}{1}\right|=1,
\]
and the powers of $A_{q+1}$ and $\lambda_{q+1}$ cancel. Thus the only coefficient multiplying $\|E_q\|_{\dot{\mathbb{A}}_{x,t}^{-s}}$ is $K_{q+1}$. By the choice of the profile in Lemma~\ref{lem:boldW}, we have
\[
\hat{\phi}_k\ge0.
\]
Therefore the Fourier coefficients of $\rho_{q+1}$ are nonnegative. Moreover, the multiplier of $\mathbb P_{\le\nu_{q+1}}$ satisfies
\[
0\le m_{\le\nu_{q+1}}\le1.
\]
Hence
\[
    K_{q+1} \le \sum_{\xi_1+\cdots+\xi_k=0}\prod_{i=1}^k \hat{\rho}_{q+1}(\xi_i) = \left(\rho_{q+1}^k\right)^{\wedge}(0) = \int_\T \rho_{q+1}^k\,dx =1
\]
Thus
\begin{equation}\label{eq:E-coeff-bound}
    K_{q+1} \le 1.
\end{equation}
We now return to \eqref{eq:abs_sum_expansion}. Fix $1\leq m\leq k-1$. We split the mixed sum according to whether
\[
    \xi_1+\cdots+\xi_m=0
\]
or not.

First suppose
\[
    \xi_1+\cdots+\xi_m\neq0.
\]
Using Petree's inequality
\[
    |\alpha+\beta|^{-s}\mathbf 1_{\alpha+\beta\neq0}
    \lesssim_s
    |\alpha|^{-s}\langle \beta\rangle^s,
    \qquad \alpha\neq0,
\]
with
\[
    \alpha=\xi_1+\cdots+\xi_m,
    \qquad
    \beta=\xi_{m+1}+\cdots+\xi_k,
\]
we obtain
\[
\begin{aligned}
&\sum_{\substack{\xi_1+\cdots+\xi_k\neq0\\ \xi_1+\cdots+\xi_m\neq0}}
\left\Vert
\prod_{i=1}^m \hat{w}_{q+1}(t,\xi_i)
\prod_{i=m+1}^k \hat{u}_q(t,\xi_i)
\right\Vert_{C_t^0}
|\xi_1+\cdots+\xi_k|^{-s}
\\
&\lesssim_s
\sum_{\substack{\xi_1,\ldots,\xi_m\\ \xi_1+\cdots+\xi_m\neq0}}
\left\Vert
\prod_{i=1}^m \hat{w}_{q+1}(t,\xi_i)
\right\Vert_{C_t^0}
|\xi_1+\cdots+\xi_m|^{-s}
\\
&\qquad\qquad\times
\sum_{\xi_{m+1},\ldots,\xi_k}
\left\Vert
\prod_{i=m+1}^k \hat{u}_q(t,\xi_i)
\right\Vert_{C_t^0}
\left\langle \xi_{m+1}+\cdots+\xi_k\right\rangle^s
\\
&=
AF(w_{q+1},m)
\sum_{\xi_{m+1},\ldots,\xi_k}
\left\Vert
\prod_{i=m+1}^k \hat{u}_q(t,\xi_i)
\right\Vert_{C_t^0}
\left\langle \xi_{m+1}+\cdots+\xi_k\right\rangle^s.
\end{aligned}
\]
Since $u_q$ is a fixed smooth stage-$q$ function with finite Fourier support, the final sum is bounded by a stage-$q$ constant. Hence, by Lemma \ref{lem:horrible1}, we get
\[
\begin{aligned}
&\sum_{\substack{\xi_1+\cdots+\xi_k\neq0\\ \xi_1+\cdots+\xi_m\neq0}}
\left\Vert
\prod_{i=1}^m \hat{w}_{q+1}(t,\xi_i)
\prod_{i=m+1}^k \hat{u}_q(t,\xi_i)
\right\Vert_{C_t^0}
|\xi_1+\cdots+\xi_k|^{-s} \lesssim_{q,m,s}
\lambda_{q+1}^{(1-\epsilon_{q+1})(\frac mk-1)}
A_{q+1}^{m/k}.
\end{aligned}
\]

Now suppose
\[
    \xi_1+\cdots+\xi_m=0.
\]
Since the full output frequency satisfies
\[
    \xi_1+\cdots+\xi_k\neq0,
\]
we must have
\[
    \xi_{m+1}+\cdots+\xi_k\neq0.
\]
Therefore
\[
    |\xi_1+\cdots+\xi_k|^{-s}
    =
    |\xi_{m+1}+\cdots+\xi_k|^{-s}.
\]
Hence
\[
\begin{aligned}
&\sum_{\substack{\xi_1+\cdots+\xi_k\neq0\\ \xi_1+\cdots+\xi_m=0}}
\left\Vert
\prod_{i=1}^m \hat{w}_{q+1}(t,\xi_i)
\prod_{i=m+1}^k \hat{u}_q(t,\xi_i)
\right\Vert_{C_t^0}
|\xi_1+\cdots+\xi_k|^{-s}
\\
&\leq
\sum_{\xi_1+\cdots+\xi_m=0}
\left\Vert
\prod_{i=1}^m \hat{w}_{q+1}(t,\xi_i)
\right\Vert_{C_t^0}
\\
&\qquad\qquad\times
\sum_{\xi_{m+1}+\cdots+\xi_k\neq0}
\left\Vert
\prod_{i=m+1}^k \hat{u}_q(t,\xi_i)
\right\Vert_{C_t^0}
|\xi_{m+1}+\cdots+\xi_k|^{-s}
\\
&=
ZF(w_{q+1},m)
\sum_{\xi_{m+1}+\cdots+\xi_k\neq0}
\left\Vert
\prod_{i=m+1}^k \hat{u}_q(t,\xi_i)
\right\Vert_{C_t^0}
|\xi_{m+1}+\cdots+\xi_k|^{-s}.
\end{aligned}
\]
Again, the final sum is bounded by a stage-$q$ constant. Therefore, by Lemma \ref{lem:horrible2}, we get
\[
\begin{aligned}
&\sum_{\substack{\xi_1+\cdots+\xi_k\neq0\\ \xi_1+\cdots+\xi_m=0}}
\left\Vert
\prod_{i=1}^m \hat{w}_{q+1}(t,\xi_i)
\prod_{i=m+1}^k \hat{u}_q(t,\xi_i)
\right\Vert_{C_t^0}
|\xi_1+\cdots+\xi_k|^{-s}
\\
&\lesssim_{q,m,s}
\lambda_{q+1}^{(1-\epsilon_{q+1})(\frac mk-1)}
A_{q+1}^{m/k}.
\end{aligned}
\]
Combining the two subcases, for each $1\leq m\leq k-1$,
\begin{equation}\label{eq:mixed-prod-bound}
\begin{aligned}
&\sum_{\xi_1+\cdots+\xi_k\neq0}
\left\Vert
\prod_{i=1}^m \hat{w}_{q+1}(t,\xi_i)
\prod_{i=m+1}^k \hat{u}_q(t,\xi_i)
\right\Vert_{C_t^0}
|\xi_1+\cdots+\xi_k|^{-s}
\\
&\lesssim_{q,m,s}
\lambda_{q+1}^{(1-\epsilon_{q+1})(\frac mk-1)}
A_{q+1}^{m/k}.
\end{aligned}
\end{equation}
Since $m\leq k-1$, the exponent
\[
    (1-\epsilon_{q+1})\left(\frac mk-1\right)
\]
is strictly negative. Thus, after $A_{q+1}$ has been fixed, we choose $\lambda_{q+1}$ sufficiently large so that the total mixed contribution in \eqref{eq:abs_sum_expansion} is bounded by $2^{-q-103}$.

It remains to estimate the pure $w_{q+1}$ contribution. For this term, we use the first estimate in Lemma \ref{lem:horrible1} with $j=k$, keeping the $E_q$-coefficient from \eqref{eq:E-coeff-pure-w} explicit. This gives
\[
\begin{aligned}
AF(w_{q+1},k)
\le\;&
K_{q+1}\|E_q\|_{\dot{\mathbb{A}}_{x,t}^{-s}}
\\
&+
C_{q,k,s}
\left(
    \lambda_{q+1}^{-s\epsilon_{q+1}}A_{q+1}
    +
    \lambda_{q+1}^{-s\epsilon_{q+1}}
    +
    A_{q+1}^{-1}
\right).
\end{aligned}
\]
By the inductive hypothesis for the error, and \eqref{eq:E-coeff-bound},
\[
    K_{q+1}\|E_q\|_{\dot{\mathbb{A}}_{x,t}^{-s}}
    \le 2^{-q-200}.
\]
Now we choose $A_{q+1}$ sufficiently large so that
\begin{equation}\label{eq:A_q+1_cond2}
    \frac{C_{q,k,s}}{A_{q+1}}
    \leq 2^{-q-104}.
\end{equation}
After $A_{q+1}$ is fixed, choose $\lambda_{q+1}$ sufficiently large so that
\begin{equation}\label{eq:pure-prod-small}
    C_{q,k,s}\lambda_{q+1}^{-s\epsilon_{q+1}}A_{q+1}
    +
    C_{q,k,s}\lambda_{q+1}^{-s\epsilon_{q+1}}
    \leq
    2^{-q-104}.
\end{equation}
Hence
\[
    AF(w_{q+1},k)
    \leq
    2^{-q-103}.
\]
Combining these estimates in \eqref{eq:abs_sum_expansion}, we obtain
\[
\begin{aligned}
&\sum_{\xi_1+\cdots+\xi_k\neq0}
\left\Vert
\hat{u}_{q+1}(t,\xi_1)\cdots \hat{u}_{q+1}(t,\xi_k)
\right\Vert_{C_t^0}
|\xi_1+\cdots+\xi_k|^{-s}
\\
&\leq
C-2^{-q-100}
+
2^{-q-103}
+
2^{-q-103}.
\end{aligned}
\]
Since
\[
    2^{-q-103}+2^{-q-103}
    =
    2^{-q-102}
    <
    2^{-q-101}
    =
    2^{-q-100}-2^{-(q+1)-100},
\]
we conclude
\[
\sum_{\xi_1+\cdots+\xi_k\neq0}
\left\Vert
\hat{u}_{q+1}(t,\xi_1)\cdots \hat{u}_{q+1}(t,\xi_k)
\right\Vert_{C_t^0}
|\xi_1+\cdots+\xi_k|^{-s}
<
C-2^{-(q+1)-100}.
\]
This proves Item \ref{i:5} at level $q+1$.

\section{Further Results}\label{sec:6}

In this section we briefly indicate some other PDEs where the methods we have presented to this point will apply \textit{mutatis mutandis}.

\subsection{Stationary \textsf{gKdV}}

In \cite{Pathak} the the third author considered the stationary \textsf{KdV} equation and proved both a flexibility and rigidity theorem. Our methods in this paper improve on this result in two regards. First, we may replace the paraproduct notion for defining the nonlinearity which was employed in \cite{Pathak} by a stationary analogue of Definition \ref{def:paras}. We will in fact prove the notion for defining the nonlinearity in Definition \ref{def:paras} is stronger than this "paraproduct notion" in Lemma \ref{lem:absolute-implies-paraproduct} in the \hyperref[sec:appendix]{appendix}. Second, and more obviously, we are able to treat stationary \textsf{gKdV} instead of being confined to \textsf{KdV}.

\begin{theorem}[\textbf{Rigidity for stationary \textsf{gKdV}}]\label{thm:stat-gkdv-rig}
Let $k\geq 2$ be an integer and let $a,b\in\mathbb R$ with $a\neq0$.
Suppose that $u\in L^k(\T)$ is real-valued and is a weak solution of the
stationary generalized \textsf{KdV} equation
\[
a\,u^{(3)}+b\,(u^k)'=0
\]
in the sense that
\begin{equation*}
    -a\int_\T u\,\psi^{(3)}\,dx
    -
    b\int_\T u^k\,\psi'\,dx
    =
    0
\end{equation*}
for every $\psi\in C^\infty(\T)$. Then $u\in C^\infty(\T)$.
\end{theorem}

\begin{proof}
Since $u\in L^k(\T)$, we have $u^k\in L^1(\T)$, so the weak formulation is
an ordinary distributional identity. It says precisely that
\[
a\,u^{(3)}+b\,(u^k)'=0
\]
in $\mathcal D'(\T)$, or equivalently
\[
\bigl(a\,u''+b\,u^k\bigr)'=0 .
\]
Hence
\begin{equation}\label{eq:integrated-stat-gkdv}
    a\,u''+b\,u^k=C
\end{equation}
in $\mathcal D'(\T)$ for some constant $C\in\mathbb R$.

Since $u^k\in L^1(\T)$, \eqref{eq:integrated-stat-gkdv} gives
\[
u''=\frac{C-bu^k}{a}\in L^1(\T).
\]
Together with $u\in L^k(\T)\subset L^1(\T)$, standard one-dimensional
periodic elliptic regularity yields
\[
u\in W^{2,1}(\T).
\]
Therefore, by the Sobolev embedding $W^{2,1}(\T)\subset C^1(\T)$, we have
\[
u\in C^1(\T).
\]
Consequently $u^k\in C^1(\T)$. Returning to
\eqref{eq:integrated-stat-gkdv}, we obtain
\[
u''=\frac{C-bu^k}{a}\in C^1(\T),
\]
and hence $u\in C^3(\T)$.

Now bootstrap. If $u\in C^m(\T)$, then $u^k\in C^m(\T)$, so
\eqref{eq:integrated-stat-gkdv} gives $u''\in C^m(\T)$, and therefore
$u\in C^{m+2}(\T)$. Starting from $u\in C^1(\T)$, this induction gives $ u\in C^{\infty}(\T)$.
\end{proof}

\begin{theorem}[\textbf{Flexibility for stationary \textsf{gKdV}}]\label{thm:stat-gkdv-flex}
Let $k\geq2$ be an integer, and let $a,b\in\mathbb R$ with
\[
ab\neq0.
\]
For every $0<\varepsilon<1$, there exists a real-valued distribution $u$ on
$\T$ such that
\[
u\in L^{k-\varepsilon}(\T)\setminus L^k(\T),
\]
the product $\mathbb P_{\neq0}(u^k)$ is well-defined by the absolute
summability criterion of Definition~\ref{def:paras}, with the obvious modifications for the stationary setting, as an element of
$\dot H^{-s}(\T)$ for all sufficiently large $s$, and $u$ solves
\[
a\,u^{(3)}+b\,(u^k)'=0
\]
in the weak singular sense
\begin{equation*}
    -a\langle u,\psi^{(3)}\rangle
    -
    b\left\langle
        \mathbb P_{\neq0}(u^k),\psi'
    \right\rangle
    =
    0
\end{equation*}
for every $\psi\in C^\infty(\T)$. Moreover, when $k\geq3$, the construction
may be arranged so that
\[
u\in H^{\frac12-\frac1k-\varepsilon}(\T).
\]
\end{theorem}
\begin{proof}
    The proof is an obvious modification of the arguments presented.
\end{proof}

\subsection{The generalized Benjamin-Ono equation}\label{sec:gBO-variant}

The same convex integration scheme also applies to the generalized
Benjamin-Ono (\textsf{gBO}) equation in the formulation used by Molinet-Tanaka in \cite{MT}.  Namely, if $\mathcal H$ denotes the periodic Hilbert transform,
\[
\widehat{\mathcal H f}(\xi)=-i\,{\rm sgn}(\xi)\hat{f}(\xi),
\qquad
{\rm sgn}(0)=0,
\]
then the $k$-generalized Benjamin-Ono equation ($k$-\textsf{gBO}) is
\begin{equation}\label{eq:gBO}
    \partial_t u-\mathcal H\partial_{xx}u
    +\sigma\partial_x(u^k)=0.
\end{equation}
Equivalently, this is the case $L_2=-\mathcal H\partial_{xx}$ and
$f(u)=\sigma u^k$ in the generalized dispersive framework of \cite{MT}.

The proof of Theorem~\ref{thm:main} is insensitive to this replacement of the
linear dispersive operator.  Indeed, in the relaxed equation for \textsf{gBO} one writes
\[
    \partial_t u-\mathcal H\partial_{xx}u
    +\sigma\partial_x(u^k)
    =
    \partial_x E .
\]
The initial relaxed solution is defined with this linear operator in place of
$\partial_{xxx}$. After that harmless change, the inductive step is the same as in the proof Theorem~\ref{thm:main}. After setting $u_{q+1}=u_q+w_{q+1}$, the only change in the error decomposition is the dispersion error.  For \textsf{gKdV} we used
\[
    \partial_{xxx}w_{q+1}
    =
    \partial_x\bigl(\partial_{xx}w_{q+1}\bigr),
    \qquad
    E_D=\partial_{xx}w_{q+1}.
\]
For \textsf{gBO} we instead have
\[
    -\mathcal H\partial_{xx}w_{q+1}
    =
    \partial_x\bigl(-\mathcal H\partial_xw_{q+1}\bigr),
\]
since $\mathcal H$ commutes with $\partial_x$.  Thus the dispersion error is
replaced by
\[
    E_D^{\rm BO}:=-\mathcal H\partial_xw_{q+1}.
\]
Moreover, the estimate of the dispersion error is easier in the \textsf{gBO} case.  The
Hilbert transform has Fourier multiplier of modulus at most one, and therefore
is bounded on the Fourier absolute spaces used in the scheme:
\[
    \|\mathcal H f\|_{\dot{\mathbb A}_{x,t}^{-s}}
    \leq
    \|f\|_{\dot{\mathbb A}_{x,t}^{-s}}.
\]
Consequently,
\[
\begin{aligned}
    \|E_D^{\rm BO}\|_{\dot{\mathbb{A}}_{x,t}^{-s}}
    &=
    \|-\mathcal H\partial_xw_{q+1}\|_{\dot{\mathbb{A}}_{x,t}^{-s}}  \\
    &\lesssim
    \sum_{\xi\neq0}
    |\xi|^{1-s}
    \|\hat{w}_{q+1}(\cdot,\xi)\|_{C_t^0}.
\end{aligned}
\]
Using
\[
    \|\hat{w}_{q+1}(\cdot,\xi)\|_{C_t^0}
    \leq
    \|w_{q+1}\|_{C_t^0L_x^1},
\]
we obtain, for $s>2$,
\[
\begin{aligned}
    \|E_D^{\rm BO}\|_{\dot{\mathbb{A}}_{x,t}^{-s}}
    &\lesssim
    \|w_{q+1}\|_{C_t^0L_x^1}
    \sum_{\xi \not=0}
    |\xi|^{1-s}  \lesssim
    \|w_{q+1}\|_{C_t^0L_x^1} .
\end{aligned}
\]
By Lemma~\ref{lem:wq+1_est},
\[
    \|w_{q+1}\|_{C_t^0L_x^1}
    \lesssim_q
    \lambda_{q+1}^{(1-\epsilon_{q+1})(\frac1k-1)}.
\]
Therefore
\[
    \|E_D^{\rm BO}\|_{\dot{\mathbb{A}}_{x,t}^{-s}}
    \lesssim_q
    \lambda_{q+1}^{(1-\epsilon_{q+1})(\frac1k-1)}.
\]
Choosing $\lambda_{q+1}$ sufficiently large at each stage makes this error as small as required.  This is strictly better than the corresponding \textsf{gKdV} dispersion estimate, where $\partial_{xxx}$ produces the error $\partial_{xx}w_{q+1}$ and hence one extra power of frequency.

Consequently the same construction gives nonunique weak singular solutions to \eqref{eq:gBO}.  In particular, there exists a nontrivial solution to \eqref{eq:gBO}, in the corresponding weak singular sense, with zero initial data, such that
\[
    u\in \bigcap_{\epsilon > 0} C_t^0L_x^{k-\epsilon}([0,1]\times\T)
\]
When $k\geq3$, the same Sobolev refinement as in Theorem~\ref{thm:main}
also gives
\[
    u\in \bigcap_{\epsilon > 0} C_t^0H_x^{\frac{1}{2}-\frac{1}{k} - \epsilon}([0,1]\times\T).
\]
Here the weak formulation is obtained from
Definition~\ref{def:weak_para_soln} by replacing the test-function term
$\partial_{xxx}\phi$ with $-\mathcal H\partial_{xx}\phi$.

\subsection{Other dispersive variants}\label{sec:other-dispersive-variants}
The construction is not tied to the precise Airy dispersion
$\partial_{xxx}$ or the dispersion term appearing in \textsf{gBO}. The same argument applies to any translation-invariant
linear dispersive operator $\mathcal L$ whose symbol can be divided by the
spatial derivative with at most polynomial loss.

More precisely, suppose that $\mathcal L$ is a skew-adjoint Fourier
multiplier of the form
\[
    \widehat{\mathcal L f}(\xi)=i\ell(\xi)\hat f(\xi),
\]
where $\ell:\mathbb Z\to\mathbb R$ is odd, $\ell(0)=0$, and for some
$r\ge0$,
\[
    \left|\frac{\ell(\xi)}{\xi}\right|
    \lesssim
    |\xi|^r,
    \qquad \xi\neq0.
\]
Equivalently, on nonzero frequencies,
\[
    \mathcal L=\partial_x\mathcal M,
\]
where $\mathcal M$ is a real Fourier multiplier of order at most $r$. Then the convex integration scheme for
\[
    \partial_tu+\mathcal L u+\sigma\partial_x(u^k)=0
\]
is unchanged except for the dispersion error.

Indeed, the relaxed equation becomes
\[
    \partial_tu+\mathcal L u+\sigma\partial_x(u^k)=\partial_xE.
\]
After setting $u_{q+1}=u_q+w_{q+1}$, the oscillation, Nash, and temporal
errors are exactly the same as in the proof of Theorem~\ref{thm:main}.  The only
replacement is
\[
    E_D=\mathcal M w_{q+1}
\]
instead of $E_D=\partial_{xx}w_{q+1}$. We obtain, for $s>r+1$,
\[
\begin{aligned}
    \|E_D\|_{\dot{\mathbb{A}}_{x,t}^{-s}}
    &\lesssim
    \sum_{\xi\neq0}
    |\xi|^{-s}
    \left|\frac{\ell(\xi)}{\xi}\right|
    \|\hat{w}_{q+1}(\cdot,\xi)\|_{C_t^0} \\
    &\lesssim
    \|w_{q+1}\|_{C_t^0L_x^1}
    \sum_{\xi \not=0}
    |\xi|^{r-s}  \\
    &\lesssim
    \|w_{q+1}\|_{C_t^0L_x^1}.
\end{aligned}
\]
Since
\[
    \|w_{q+1}\|_{C_t^0L_x^1}
    \lesssim_q
    \lambda_{q+1}^{(1-\epsilon_{q+1})(\frac1k-1)},
\]
Choosing $\lambda_{q+1}$ sufficiently large at each stage makes this error as small as required. Thus, the induction closes exactly as before.

This includes the generalized dispersive class considered by Molinet-Tanaka
\cite{MT}, where the linear operator $L_{\alpha+1}$ has Fourier
symbol $-ip_{\alpha+1}(\xi)$ with
\[
    p_{\alpha+1}'(\xi)\sim \xi^\alpha,
    \qquad
    p_{\alpha+1}''(\xi)\sim \xi^{\alpha-1},
    \qquad
    1\le\alpha\le2.
\]
In that case
\[
    p_{\alpha+1}(\xi)\sim |\xi|^\alpha\xi,
\]
so after division by $\partial_x$ the dispersion error has order
$r=\alpha\le2$.  Therefore the estimate above is no worse than the \textsf{gKdV}
dispersion estimate.  The endpoint $\alpha=2$ gives \textsf{gKdV}, while the endpoint
$\alpha=1$ gives generalized Benjamin-Ono from Subsection~\ref{sec:gBO-variant}.
The same observation also covers mixed constant-coefficient dispersions such as
\[
    \partial_tu+
    \bigl(
        c_3\partial_{xxx}
        -c_2\mathcal H\partial_{xx}
        +c_1\partial_x
    \bigr)u
    +\sigma\partial_x(u^k)=0,
\]
and, with a larger choice of $s$, higher-order dispersive models such as
\[
    \partial_tu+\partial_x^5u+\sigma\partial_x(u^k)=0.
\]
In every case, the nonlinear part of the construction is unchanged; only the
order of the dispersion error changes.

\section{Appendix}\label{sec:appendix}
We first record the elementary fact that weak singular solutions conserve their
spatial mean.

\begin{lemma}[Conservation of the spatial mean]\label{lem:mean-conservation}
Let $u \in C_t^0 L_x^1$ be a weak singular solution of
\[
\partial_t u+\partial_{xxx}u+\sigma\partial_x(u^k)=0
\]
on $[0,1]\times\T$ in the sense of Definition~\ref{def:weak_para_soln},
with initial datum $u_0$. Then
\[
\int_{\T}u(t,x)\,dx=\int_{\T}u_0(x)\,dx
\qquad
\text{for every }t\in[0,1].
\]
Equivalently,
\[
\hat{u}(t,0)=\hat{u}_0(0)
\qquad
\text{for every }t\in[0,1].
\]
\end{lemma}

\begin{proof}
Recall
\[
\mathbb{P}_{=0,x}(u)(t):=\int_{\T}u(t,x)\,dx=\langle u(t),1\rangle_{H^s,H^{-s}}.
\]
Since $u\in C_t^0H_x^s$ and $1\in H^{-s}(\T)$, the function
$\mathbb{P}_{=0,x}(u)$ belongs to $C^0([0,1])$.

Let $\chi\in C_c^\infty([0,1))$, and use the spatially constant test
function
\[
\phi(t,x):=\chi(t).
\]
Then
\[
\partial_{xxx}\phi=0,
\qquad
\partial_x\phi=0.
\]
Therefore the dispersive term and the nonlinear term vanish in the weak
formulation. Hence Definition~\ref{def:weak_para_soln} gives
\[
\int_0^1 \langle u(t),\chi'(t)\rangle_{H^s,H^{-s}}\,dt
=
-\langle u_0,\chi(0)\rangle_{H^s,H^{-s}}.
\]
Since $\chi(t)$ is spatially constant, this is exactly
\[
\int_0^1 \mathbb{P}_{=0,x}(u)(t)\chi'(t)\,dt
=
-\mathbb{P}_{=0,x}(u)(0)\chi(0).
\]

On the other hand,
\[
\int_0^1 \chi'(t)\,dt=-\chi(0),
\]
because $\chi$ vanishes near $t=1$. Therefore
\[
\begin{aligned}
\int_0^1 \left(\mathbb{P}_{=0,x}(u)(t) - \mathbb{P}_{=0,x}(u)(0)\right)\chi'(t)\,dt
&=
\int_0^1 \mathbb{P}_{=0,x}(u)(t)\chi'(t)\,dt
-
\mathbb{P}_{=0,x}(u)(0)\int_0^1\chi'(t)\,dt \\
&=
-\mathbb{P}_{=0,x}(u)(0)\chi(0)-\mathbb{P}_{=0,x}(u)(0)(-\chi(0)) \\
&=0.
\end{aligned}
\]
Thus the distributional derivative of $\mathbb{P}_{=0,x}(u)(t)-\mathbb{P}_{=0,x}(u)(0)$ vanishes on $(0,1)$.
Hence $\mathbb{P}_{=0,x}(u)(t) - \mathbb{P}_{=0,x}(u)(0)$ is constant on $(0,1)$. Since $\mathbb{P}_{=0,x}(u)$ is continuous, the constant is zero, and hence
\[
\mathbb{P}_{=0,x}(u)(t) = \mathbb{P}_{=0,x}(u)(0)
\qquad
\text{for every }t\in[0,1].
\]
This proves the claim.
\end{proof}

\begin{remark}
In the special case of Theorem~\ref{thm:main}, the initial data is
$u_0=0$. Therefore Lemma~\ref{lem:mean-conservation} gives
\[
\int_{\T}u(t,x)\,dx=0
\qquad
\text{for every }t\in[0,1]
\]
for the solution we construct.
\end{remark}

Next we show absolute Fourier summability implies the paraproduct summability criterion utilized in \cite{ABGN}, \cite{CH}, \cite{Gismondi}, and \cite{GR}.

\begin{lemma}
\label{lem:absolute-implies-paraproduct}
Let $\alpha\in\mathbb R$, and let $f_1,\ldots,f_k\in\mathcal D'(\T)$.
Assume
\begin{equation*}
\sum_{\xi_1+\cdots+\xi_k\neq0}
|\hat{f}_1(\xi_1)|\cdots|\hat{f}_k(\xi_k)|
|\xi_1+\cdots+\xi_k|^\alpha
<\infty .
\end{equation*}
Assume moreover that the Littlewood-Paley resolution reconstructs every
input frequency which contributes to the above sum. Equivalently, either each
$f_i$ is mean-zero and the homogeneous resolution satisfies
\[
\sum_j \varphi_j(\xi)=1
\qquad\text{for every }\xi\neq0,
\]
or the resolution includes the zero mode and satisfies
\[
\sum_j \varphi_j(\xi)=1
\qquad\text{for every }\xi\in\mathbb Z.
\]
Then the paraproduct expansion
\begin{equation*}
\sum_{j_1,\ldots,j_k}
\mathbb P_{\neq0}\Bigl(
\mathbb P_{2^{j_1}}f_1\cdots \mathbb P_{2^{j_k}}f_k
\Bigr)
\end{equation*}
converges absolutely in the homogeneous weighted Wiener space
\[
\dot{\mathbb{A}}^\alpha(\T)
=
\left\{
g\in\mathcal D'(\T):
\|g\|_{\dot{\mathbb{A}}^\alpha}
:=
\sum_{\eta\neq0}|\eta|^\alpha|\hat{g}(\eta)|<\infty
\right\}.
\]
In particular it converges absolutely in $\dot H^\alpha(\T)$. Moreover its
sum agrees with the absolutely defined Fourier product
\begin{equation*}
\mathbb P_{\neq0}(f_1\cdots f_k)
:=
\sum_{\xi_1+\cdots+\xi_k\neq0}
\hat{f}_1(\xi_1)\cdots\hat{f}_k(\xi_k)
e^{2\pi i(\xi_1+\cdots+\xi_k)x}.
\end{equation*}
\end{lemma}

\begin{proof}
Recall from Definition \ref{def:projs}, $\varphi_j$ denotes the Fourier multiplier of $\mathbb P_{2^j}$. We use the finite overlap
property
\begin{equation*}
\sum_j |\varphi_j(\xi)|\le 2
\end{equation*}
for every frequency $\xi$ covered by the resolution.

Fix $J=(j_1,\ldots,j_k)$. For $\eta\neq0$, the $\eta$-Fourier coefficient
of
\[
\mathbb P_{\neq0}\Bigl(\mathbb P_{2^{j_1}}f_1\cdots \mathbb P_{2^{j_k}}f_k\Bigr)
\]
is
\[
\sum_{\xi_1+\cdots+\xi_k=\eta}
\varphi_{j_1}(\xi_1)\hat{f}_1(\xi_1)\cdots
\varphi_{j_k}(\xi_k)\hat{f}_k(\xi_k).
\]
Hence
\begin{align*}
&\left\|
\mathbb P_{\neq0}\Bigl(\mathbb P_{2^{j_1}}f_1\cdots \mathbb P_{2^{j_k}}f_k\Bigr)
\right\|_{\dot{\mathbb{A}}^\alpha}
\notag\\
&\quad \le
\sum_{\xi_1+\cdots+\xi_k\neq0}
|\varphi_{j_1}(\xi_1)|\cdots |\varphi_{j_k}(\xi_k)|
|\hat{f}_1(\xi_1)|\cdots|\hat{f}_k(\xi_k)|
|\xi_1+\cdots+\xi_k|^\alpha .
\end{align*}
Summing over $j_1,\ldots,j_k$ and applying Tonelli's theorem gives
\begin{align*}
&\sum_{j_1,\ldots,j_k}
\left\|
\mathbb P_{\neq0}\Bigl(\mathbb P_{2^{j_1}}f_1\cdots \mathbb P_{2^{j_k}}f_k\Bigr)
\right\|_{\dot{\mathbb{A}}^\alpha}
\\
&\quad\le
\sum_{\xi_1+\cdots+\xi_k\neq0}
|\hat{f}_1(\xi_1)|\cdots|\hat{f}_k(\xi_k)|
|\xi_1+\cdots+\xi_k|^\alpha
\prod_{\ell=1}^k
\left(\sum_{j_\ell}|\varphi_{j_\ell}(\xi_\ell)|\right)
\\
&\quad\le
2^k
\sum_{\xi_1+\cdots+\xi_k\neq0}
|\hat{f}_1(\xi_1)|\cdots|\hat{f}_k(\xi_k)|
|\xi_1+\cdots+\xi_k|^\alpha
<\infty.
\end{align*}
Thus the paraproduct series converges absolutely in $\dot{\mathbb{A}}^\alpha$.

Since
\[
\begin{aligned}
\|\mathbb P_{\neq0}\Bigl(\mathbb P_{2^{j_1}}f_1\cdots \mathbb P_{2^{j_k}}f_k\Bigr)\|_{\dot H^\alpha}
&=
\left(
\sum_{\eta\neq0}|\eta|^{2\alpha}\left|\mathbb P_{\neq0}\Bigl(\mathbb P_{2^{j_1}}f_1\cdots \mathbb P_{2^{j_k}}f_k\Bigr)^\wedge(\eta)\right|^2
\right)^{1/2}\\
&\le
\sum_{\eta\neq0}|\eta|^\alpha\left|\mathbb P_{\neq0}\Bigl(\mathbb P_{2^{j_1}}f_1\cdots \mathbb P_{2^{j_k}}f_k\Bigr)^\wedge(\eta)\right|\\
&=
\|\mathbb P_{\neq0}\Bigl(\mathbb P_{2^{j_1}}f_1\cdots \mathbb P_{2^{j_k}}f_k\Bigr)\|_{\dot{\mathbb{A}}^\alpha},
\end{aligned}
\]
the same series also converges absolutely in $\dot H^\alpha$.

It remains to identify the limit. Let $I_N$ be an increasing sequence of
finite index sets exhausting the Littlewood-Paley resolution, and define
\[
S_N
:=
\sum_{j_1,\ldots,j_k\in I_N}
\mathbb P_{\neq0}\Bigl(\mathbb P_{2^{j_1}}f_1\cdots \mathbb P_{2^{j_k}}f_k\Bigr).
\]
For $\eta\neq0$,
\[
\hat{S}_N(\eta)
=
\sum_{\xi_1+\cdots+\xi_k=\eta}
\prod_{\ell=1}^k
\left(\sum_{j_\ell\in I_N}\varphi_{j_\ell}(\xi_\ell)\right)
\hat{f}_1(\xi_1)\cdots\hat{f}_k(\xi_k).
\]
For every contributing input frequency, the reconstruction assumption gives
\[
\sum_j \varphi_j(\xi)=1.
\]
Also,
\[
\left|\sum_{j\in I_N}\varphi_j(\xi)\right|
\le
\sum_j |\varphi_j(\xi)|
\le 2.
\]
Therefore, by dominated convergence and the absolute summability assumption,
\[
\hat{S}_N(\eta)\to
\sum_{\xi_1+\cdots+\xi_k=\eta}
\hat{f}_1(\xi_1)\cdots\hat{f}_k(\xi_k)
\]
for every $\eta\neq0$. Moreover, the same majorant gives convergence in
$\dot{\mathbb{A}}^\alpha$. Indeed,
\[
\|S_N-\mathbb P_{\neq0}(f_1\cdots f_k)\|_{\dot{\mathbb{A}}^\alpha}\to0
\]
by dominated convergence applied to the nonnegative summable majorant
\[
2^k
|\hat{f}_1(\xi_1)|\cdots|\hat{f}_k(\xi_k)|
|\xi_1+\cdots+\xi_k|^\alpha.
\]
Hence the paraproduct sum agrees with the absolutely defined Fourier product.
\end{proof}
We now compare our absolute Fourier product with Christ's cutoff formulation. We use Christ's terminology from \cite[Definitions~2.1-2.3]{Christ1}: a Fourier cutoff sequence is a sequence of finite-support Fourier multipliers $P_N$ whose multipliers are uniformly bounded and converge pointwise to $1$; a nonlinear expression exists in the cutoff sense if the corresponding regularized nonlinearities converge in distributions for every such sequence.

\begin{lemma}\label{lem:absolute-implies-christ-cutoff}
    Let $u\in C_t^0\mathcal D'_x([0,1] \times \T)$, let $k\ge2$, and let $s>0$. Assume
    \[
    \sum_{\xi_1+\cdots+\xi_k\neq 0} \left\|\hat u(\cdot,\xi_1)\cdots\hat u(\cdot,\xi_k)\right\|_{C_t^0}|\xi_1+\cdots+\xi_k|^{-s}<\infty.
    \]
    Define $\mathbb P_{\neq0}(u^k)$ by
    \[
    \left(\mathbb P_{\neq0}(u^k)\right)^{\wedge}(t,\eta) = \sum_{\xi_1+\cdots+\xi_k=\eta} \hat u(t,\xi_1)\cdots \hat u(t,\xi_k), \qquad \eta\neq 0,
    \]
    and set its zero Fourier mode equal to $0$. Then
    \[
    \mathbb P_{\neq0}(u^k) \in C_t^0\dot{\A}^{-s}_x([0,1] \times \T) \hookrightarrow C_t^0\dot H^{-s}_x([0,1] \times \T).
    \]
    Moreover, for every Fourier cutoff sequence $P_N$ in the sense of \cite[Definition~2.1]{Christ1},
    \[
    \partial_x \bigl((P_Nu)^k\bigr)\longrightarrow\partial_x\mathbb P_{\neq0}(u^k) \qquad \text{in } C_t^0\dot H^{-s-1}_x([0,1] \times \T).
    \]
    Consequently, $\partial_x(u^k)$ exists in Christ's cutoff sense and equals
    \[
    \partial_x\mathbb P_{\neq0}(u^k)
    \]
    in the sense of distributions. In particular, if
    \[
    \partial_t u+ \partial_{xxx}u+\sigma\partial_x\mathbb P_{\neq0}(u^k)=0
    \]
    in $\mathcal C_t^0 \mathcal{D}'_x([0,1]\times\T)$, then $u$ is a solution in Christ's extended sense with nonlinear term given by $\partial_x \mathbb P_{\neq0}(u^k)$.
\end{lemma}

\begin{proof}
    The absolute summability assumption gives
    \[
    \begin{aligned}
        \|\mathbb P_{\neq0}(u^k)\|_{C_t^0\dot{\A}_x^{-s}} &\le \sum_{\xi_1+\cdots+\xi_k\neq0}\left\|\hat u(\cdot,\xi_1)\cdots \hat u(\cdot,\xi_k)\right\|_{C_t^0}|\xi_1+\cdots+\xi_k|^{-s} <\infty.
    \end{aligned}
    \]
    The embedding into $C_t^0\dot H^{-s}_x$ follows from $\ell^1\hookrightarrow \ell^2$.
    Let $P_N$ be a Fourier cutoff sequence in Christ's sense with multiplier $m_N$. For $\eta\neq0$,
    \[
    \left(\mathbb P_{\neq0}((P_Nu)^k)\right)^{\wedge}(t,\eta) = \sum_{\xi_1+\cdots+\xi_k=\eta} \prod_{i=1}^k m_N(\xi_i) \prod_{i=1}^k \hat u(t,\xi_i).
    \]
    Hence
    \[
    \begin{aligned}
        &\left\| \mathbb P_{\neq0}((P_Nu)^k) -\mathbb P_{\neq0}(u^k)\right\|_{C_t^0\dot{\A}_x^{-s}} \\
        &\qquad\le \sum_{\xi_1+\cdots+\xi_k\neq0}\left|\prod_{i=1}^k m_N(\xi_i)-1\right|\left\|\prod_{i=1}^k \hat u(\cdot,\xi_i)\right\|_{C_t^0}|\xi_1+\cdots+\xi_k|^{-s}.
    \end{aligned}
    \]
    For each fixed tuple $(\xi_1,\ldots,\xi_k)$, we have 
    \[
    \prod_{i=1}^k m_N(\xi_i)\to1.
    \]
    Since the multipliers $m_N$ are uniformly bounded, the summand is dominated by a constant multiple of the absolutely summable series above. Thus dominated convergence gives
    \[
    \mathbb P_{\neq0}((P_Nu)^k) \to \mathbb P_{\neq0}(u^k)
    \qquad \text{in }C_t^0\dot{\mathbb A}^{-s}_x\subset C_t^0\dot H^{-s}_x.
    \]
    Finally, $\partial_x$ annihilates the zero spatial mode, so
    \[
    \partial_x((P_Nu)^k) = \partial_x\mathbb P_{\neq0}((P_Nu)^k).
    \]
    Applying
    \[
    \partial_x:C_t^0\dot H^{-s}\to C_t^0\dot H_x^{-s-1}
    \]
    to the convergence above gives
    \[
    \partial_x((P_Nu)^k)\to \partial_x\mathbb P_{\neq0}(u^k)
    \qquad \text{in }C_t^0\dot H_x^{-s-1}.
    \]
    This is precisely the existence of the nonlinearity in Christ's cutoff sense, and the final assertion follows from \cite[Definition~2.3]{Christ1}.
\end{proof}

We finally record the fact that the absolute Fourier summability is stronger than the classical weak notion of solution when the nonlinearity lies in $C_t^0L_x^1$.

\begin{proposition}\label{prop:singular-implies-classical}
    Let $k \ge 2$, and suppose $s$ is such that
    \[
    H^s(\T) \hookrightarrow L^k(\T)
    \]
    for instance $s \ge \frac12-\frac1k$. Let 
    \[
    u \in C_t^0 H^s_x([0,1] \times \T)
    \]
    be a weak singular solution of ~\eqref{eq:gkdv} in the sense of Definition ~\ref{def:weak_para_soln}, with initial data $u_0$. Then $u$ is also a classical distributional weak solution. More precisely, for every $\phi \in C_c^\infty([0,1)\times \T)$,
    \[
    \int_0^1\int_\T \left( u\partial_t\phi+u\partial_{xxx}\phi+\sigma u^k\partial_x\phi\right) \,dx\,dt = -\int_\T u_0(x)\phi(0,x)\,dx,
    \]
    where $u^k$ has the usual meaning as a $C_t^0 L_x^1([0,1] \times \T)$ function.
\end{proposition}

\begin{proof}
    Since $H^s(\T) \hookrightarrow L^k(\T)$, we have
    \[
    u\in C_t^0L_x^k, \qquad u^k \in C_t^0L_x^1.
    \]
    Thus the usual product $u^k$ is a well-defined spacetime distribution.

    We first show that the product defined in Definition ~\ref{def:paras} agrees with the usual product after projecting away from the zero mode. Let $P_N =\mathbb P_{\le N}$ be the standard smooth Fourier cutoff. Since $P_N \to I$ strongly on $L^k(\T)$, the operators $P_N$ are uniformly bounded on $L^k(\T)$, and $u([0,1],\cdot)$ is compact in $L^k(\T)$, we have
    \[
    P_N u\to u \qquad \text{in } C_t^0L_x^k.
    \]
    We claim that, as a consequence, we also have
    \[
    (P_Nu)^k \to u^k \qquad \text{in }C_t^0L_x^1
    \]
    Indeed, by the pointwise inequality
    \[
    |a^k - b^k| \lesssim_k (|a|^{k-1} +|b|^{k-1})|a-b|
    \]
    and H\"older's inequality we get
    \[
    \begin{aligned}
        \|(P_Nu)^k-u^k\|_{C_t^0L_x^1} &\lesssim_k \left(\|P_Nu\|_{C_t^0L_x^k}^{k-1} + \|u\|_{C_t^0L_x^k}^{k-1} \right) \|P_Nu-u\|_{C_t^0L_x^k}.
    \end{aligned}
    \]
    The cutoffs $P_N$ are uniformly bounded on $L^k(\T)$\footnote{The statement of Lemma \ref{lem:proj} can be suitably modified to handle the case when the frequency cutoffs are not exactly dyadic integers.}, and $P_Nu\to u$ in $C_t^0L_x^k$. Hence
    the claim is established. For every nonzero frequency $\eta$, the Fourier coefficient of the smooth product $(P_Nu)^k$ is given by the finite convolution formula
    \[
    \left((P_Nu)^k\right)^{\wedge}(t,\eta) = \sum_{\xi_1+\cdots+\xi_k=\eta}m_N(\xi_1)\cdots m_N(\xi_k) \hat u(t,\xi_1)\cdots \hat u(t,\xi_k),
    \]
    where $m_N$ is the multiplier defining $P_N$. Since $\mathbb P_{\neq0}(u^k)$ is absolutely summable in the sense of Definition ~\ref{def:paras}, dominated convergence in this convolution gives
    \[
    \left((P_Nu)^k\right)^{\wedge}(\cdot,\eta) \to \left(\mathbb P_{\neq0}(u^k)\right)^{\wedge}(\cdot,\eta) \qquad \text{in }C_t^0
    \]
    On the other hand, the convergence $(P_Nu)^k \to u^k$ in $C_t^0L_x^1$ gives
    \[
    \left((P_Nu)^k\right)^{\wedge}(\cdot,\eta) \to \left(u^k\right)^{\wedge}(\cdot,\eta) \qquad \text{in } C_t^0.
    \]
    Hence, for every $\eta \neq 0$,
    \[
    \left(\mathbb P_{\neq 0}(u^k)\right)^{\wedge}(\cdot,\eta) = \left(u^k\right)^{\wedge}(\cdot,\eta).
    \]
    Therefore $\mathbb P_{\neq0}(u^k)$ agrees, as a distribution, with the projection onto nonzero spatial frequencies of the usual product $u^k$.

    Now take $\phi \in C_c^\infty([0,1)\times \T)$. Since $\partial_x \phi(t,\cdot)$ has zero spatial mean for every $t$, the zero mode of $u^k$ does not contribute. Therefore
    \[
    \int_0^1\left\langle \mathbb P_{\neq0}(u^k), \partial_x \phi \right \rangle\,dt = \int_0^1\int_\T u^k \partial_x \phi\,dx\,dt.
    \]
    The linear pairings in Definition ~\ref{def:weak_para_soln} are ordinary integrals because $u\in C_t^0H^s_x \subset C^0_tL^1_x$ on the torus. Also $u_0\in H^s(\T)\subset L^1(\T)$, so the initial pairing is the ordinary integral against $\phi(0,\cdot)$. Substituting these identities into Definition~\ref{def:weak_para_soln} gives exactly the displayed classical weak formulation.
\end{proof}

\noindent\textsc{Department of Mathematics, Purdue University, West Lafayette, IN, USA.}
\vspace{.03in}
\newline\noindent\textit{Email address}:
\href{mailto:ngismond@purdue.edu}{ngismond@purdue.edu}.
\newline\noindent\textit{Email address}:
\href{mailto:pathak30@purdue.edu}{pathak30@purdue.edu}.

\noindent\textsc{Department of Mathematics, Columbia University, New York, NY, USA.}
\vspace{.03in}
\newline\noindent\textit{Email address}:
\href{mailto:km4046@columbia.edu}{km4046@columbia.edu}.

\noindent\textsc{Simion Stoilow Institute of Mathematics of the Romanian Academy, Calea Grivitei Street, no. 21, 010702 Bucharest, Romania.}
\newline\noindent\textit{Email address}: \href{mailto:alexandru.radu@imar.ro}{alexandru.radu@imar.ro}.
\end{document}